\newif\ifsubmitted

\submittedfalse

\ifsubmitted
\documentclass[11pt]{elsarticle}
\else
\documentclass[11pt,a4paper,reqno]{amsart}
\fi

\usepackage[latin1]{inputenc} \usepackage[T1]{fontenc} \usepackage{lmodern}
\usepackage{graphicx,a4wide,amsmath,amssymb,amsfonts,amsbsy,xy}

\usepackage{pgf,color}
\usepackage{pgfarrows}

\usepackage[colorlinks=true,linkcolor=red,citecolor=red,bookmarks=false,pdftex]{hyperref}
\hypersetup{
  pdfauthor   = {J.-M. Couveignes and R. Lercier},
  pdftitle    = {Elliptic periods and primality proving},
  pdfsubject  = {},
  pdfkeywords = {}
}

\newcommand{\trinom}[4]{{\left(\begin{array}{c}{#1}\\ #2 ,  #3 ,
      #4\end{array}\right)}}

\xyoption{all}

\newcommand{\dvec}[1]{\overrightarrow{#1}}

\definecolor{Grey}{rgb}{.5,.5,.5}

\definecolor{Blue}{rgb}{.0,.0,0.9}
\definecolor{LightBlue1}{rgb}{.2,.4,0.9}
\definecolor{LightBlue2}{rgb}{.3,.5,0.9}
\definecolor{LightBlue3}{rgb}{.4,.6,0.9}
\definecolor{LightBlue4}{rgb}{.5,.7,.9}
\definecolor{LightBlue5}{rgb}{.6,.8,.9}
\definecolor{LightBlue6}{rgb}{.7,.9,.9}

\definecolor{Red}{rgb}{.9,.0,.0}
\definecolor{LightRed1}{rgb}{0.9,.2,.4}
\definecolor{LightRed2}{rgb}{0.9,.3,.5}
\definecolor{LightRed3}{rgb}{0.9,.4,.6}
\definecolor{LightRed4}{rgb}{.9,.5,.7}
\definecolor{LightRed5}{rgb}{.9,.6,.8}
\definecolor{LightRed6}{rgb}{.9,.7,.9}

\xyoption{all}
\usepackage[english]{babel}

\newcounter{noalgo}[section]
\newdimen\indentalgo
\newdimen\indentalgodec\indentalgo=0.0mm\indentalgodec=10mm

\newcommand{\If}{\advance\indentalgo by \indentalgodec {\bf if }}

\newcommand{\For}{\global\advance\indentalgo by \indentalgodec {\bf for }}

\newcommand{\Endindent}{\global\advance\indentalgo by -\indentalgodec}

\newdimen\decalage \decalage=0.5cm
\newcounter{algo} 

\let\set\mathbb
\def\<<{\leavevmode
  \raise0.28ex\hbox{$\scriptscriptstyle\langle\!\langle$}\nobreak
  \hskip -.6pt plus.3pt minus.2pt\,}
\def\>>{\,\nobreak\hskip -.6pt plus.3pt minus.2pt
  \raise0.28ex\hbox{$\scriptscriptstyle\rangle\!\rangle$}}

\def\HT{{\hat T}}
\def\hu{{\hat u}}
\def\htheta{{\hat \theta}}
\def\hsigma{{\hat \sigma}}

\def\vvv{{\overrightarrow{v}}}
\def\vvvone{{\overrightarrow{v_1}}}
\def\vvvtwo{{\overrightarrow{v_2}}}

\def\vuNinv{{\overrightarrow{u_N}^{(-1)}}}
\def\vuN{{\overrightarrow{u_N}}}
\def\veinv{{\overrightarrow{e}^{(-1)}}}

\def\vegot{{\dvec{\egot}}}
\def\vigot{{\dvec{\igot}}}
\def\vhigot{{\dvec{\higot}}}
\def\valpha{{\dvec{\alpha}}}
\def\vgamma{{\dvec{\gamma}}}
\def\viota{{\dvec{\iota}}}
\def\hiota{{\hat \iota}}

\def\hviota{{\dvec{\hiota}}}
\def\vhiota{{\dvec{\hiota}}}
\def\vbeta{{\dvec{\beta}}}

\def\vxN{{\dvec{x_N}}}

\def\hu{{\hat u}}

\def\Res{\mathop{\rm{Res}}\nolimits }

\def\Ker{\mathop{\rm{Ker}}\nolimits }

\def\Spec{\mathop{\rm{Spec}}\nolimits }
\def\Tr{\mathop{\rm{Tr}}\nolimits }

\def\FF{{\set F}}
\def\bFF{{\overline {\set F}}}

\def\Fp{{\FF _p}}

\def\bK{{\bf K  }}
\def\bM{{\bf M  }}
\def\bbK{\overline{\bf K  }}
\def\bL{{\bf L  }}
\def\bbL{\overline {\bf L  }}
\def\bL{{\bf L  }}

\def\bG{{\bf  G }}

\def\NN{{\set N}}

\def\ZZ{{\set Z}}
\def\agot{{\mathfrak a}}
\def\Fgot{{\mathfrak F}}
\def\Ngot{{\mathfrak N}}
\def\egot{{\mathfrak e}}
\def\ve{{\dvec{e}}}
\def\xgot{{\mathfrak x}}
\def\ygot{{\mathfrak y}}
\def\tgot{{ t}}
\def\ngot{{\mathfrak n}}
\def\igot{{\mathfrak u}}
\def\higot{\hat{\mathfrak u}}

\def\dmin{{\rm {d_{min}}}}
\def\dmax{{\rm {d_{max}}}}
\def\bgot{{\mathfrak b}}
\def\cA{{\mathcal A}}

\def\cG{{\mathcal G}}

\def\cI{{\mathcal I}}
\def\cJ{{\mathcal J}}
\def\cK{{\mathcal K}}
\def\cL{{\mathcal L}}

\def\cO{{\mathcal O}}

\def\cQ{{\mathcal Q}}
\def\cR{{\mathcal R}}
\def\cS{{\mathcal S}}
\def\cT{{\mathcal T}}
\def\cU{{\mathcal U}}
\def\cV{{\mathcal V}}
\def\cW{{\mathcal W}}

\def\pgot{{\mathfrak p}}

\newtheorem{lemma}{Lemma}

\newtheorem{theorem}{Theorem}
\newtheorem{corollary}{Corollary}
\newtheorem{definition}{Definition}

\ifsubmitted
\newproof{proof}{Proof}
\fi

\usepackage{enumitem}

\begin{document}

\ifsubmitted
\else

\title[Elliptic periods and primality proving]{Elliptic periods and primality proving\\(extented version)}

\author{Jean-Marc Couveignes}%
\address{Institut de Math\'ematiques de Toulouse, Universit\'e de Toulouse et
  CNRS, D{\'e}partement de Math{\'e}matiques et Informatique, Universit{\'e}
  Toulouse 2 le Mirail, 5 all{\'e}es Antonio Machado, F-31058 Toulouse Cedex
  9, France.}%
\email{jean-marc.couveignes@math.univ-toulouse.fr}

\author{Tony Ezome}%
\email{tony.ezome@math.univ-toulouse.fr}

\author{Reynald Lercier}%
\address{ DGA/C\'ELAR, La Roche Marguerite, F-35174
  Bruz Cedex, France.
}
\address{
  Institut de recherche math\'ematique de Rennes, Universit\'e de
  Rennes 1 et CNRS, Campus de Beaulieu, F-35042 Rennes Cedex, France.
}
\email{reynald.lercier@m4x.org}

\thanks{%
    Research supported by the ``D{\'e}l{\'e}gation G{\'e}n{\'e}rale pour
    l'Armement'' and by the ``Agence Nationale de la Recherche'' (project
    ALGOL)} 

\date{\today}

\maketitle

\fi

\begin{abstract}
We construct extension rings with fast arithmetic
using isogenies between elliptic curves.
As an application, we give an elliptic version of the AKS primality criterion.
\end{abstract}

\ifsubmitted
\else
\bibliographystyle{plain}
\fi


\section{Introduction}\label{section:intro}

Classical Kummer theory  considers  binomials of the
form $x^d-\alpha$ where $d\geqslant  2$ is an integer and
$\alpha$ is a unit  in a  (commutative and  unitary) ring  containing
a primitive $d$-th root of unity $\zeta$. 
The associated $R$-algebra $S=R[x]/(x^d-\alpha)$ has shown to be extremely
useful, including in very recent algorithmic applications such as integer
factoring and discrete logarithm computation \cite{JouxLercier2}, primality 
proving \cite{AKS,Berri}, fast polynomial factorization and composition 
\cite{UK},
low complexity normal basis \cite{MOVW, GL, AVS} of field extensions
and ring extensions \cite{LL03}. 

Part
of  this computational relevance is due to the purely algebraic 
properties of $S$: a finite free \'etale  $R$-algebra of rank $d$, endowed with
an  $R$-automorphism $\sigma : x\mapsto \zeta x$ such that $R$ is the ring
of invariants by $\sigma $ in $S$ (see Section~\ref{subsection:cyclic}). However, there are more geometric
properties involved. For example, we can define the degree of a non-zero
class
in $R[x]/(x^d-\alpha)$ to be the smallest degree of non-zero polynomials
in this class.   This degree is subadditive and  invariant
by the automorphism $\sigma$. To understand this, it
is sensible  to  introduce the multiplicative group 
$\bG_m = \Spec(R[x,{1}/{x}])$  over
$R$ and the multiplication by $d$ isogeny $[d] : \bG_m\rightarrow \bG_m$.
Then $x=\alpha$  defines  a section $A$ 
of $\bG_m\rightarrow \Spec(R)$ and
$S$ can be seen as 
 the residue ring at $\Fgot_A=[d]^{-1}(A)$. The kernel of $[d]$ is the disjoint
union of $d$ sections in $\bG_m(R)$. Let $T$ be the one defined by $x=\zeta$. Translation
by $T$ defines an automorphism of $\bG_m$ that stabilizes 
$\Fgot_A$. One can  then  view elements in $S$
as congruence classes of functions on $\bG_m$ modulo $\Fgot_A$. 

The main restriction of classical Kummer
theory  is that not every ring $R$ has a primitive
$d$-th root of unity. One may  look for an auxiliary extension  $R'\supset
R$ that contains such a primitive root, but this may result in many
complications and a great loss of efficiency.
Another approach, already experimented in  the context
of normal bases \cite{CL}  for finite fields extensions, consists in
replacing the multiplicative group $\bG_m$ by some
well chosen elliptic curve $E$ over $R$. We then look for a section
$T\in E(R)$ of exact order $d$. Because elliptic curves are many,
we increase our chances to find such a section. We call the resulting
algebra $S$ a ring of {\it elliptic periods} because of the strong analogy
with classical Gauss periods.  

The first half
of the present  work is  devoted to the explicit study of Kummer theory
of elliptic curves and, more specifically, to the algebraic
and algorithmic description of the residue algebras constructed
as sketched above. The resulting elliptic functions and  equations
are not quite as simple as binomials. Still they can be described
very  explicitly  and quickly,  e.g. in quasi-linear time in the
degree $d$. The geometric situation is  summarized by Theorem~\ref{th:formules} and the $R$-algebra  $S$ of elliptic periods 
is described by Theorem~\ref{th:tensor}.
The second half of the paper proposes an elliptic version of the AKS
primality criterion. A general, context free, primality criterion in the
style of Berrizbeitia is first given in
Theorem~\ref{th:AKSgen}. This criterion involves an $R$-algebra 
$S$ where $R=\ZZ/n\ZZ$ and $n$ is the integer to be tested for primality.
If we take  $S$ to be $R[x]/(x^d-\alpha)$, we recover results by 
Berrizbeitia and his followers. If we take $S$ to be a ring of elliptic
periods, we obtain the  elliptic  primality criterion 
of Corollary~\ref{th:ellaks}.

\medskip 
While  the proof of Corollary~\ref{th:ellaks} uses the results  in
Section~\ref{section:EP}, much of Section~\ref{section:AKSTEST}
is independent of Section~\ref{section:EP}. Readers only interested
in primality proving  may skim through
Section~\ref{section:EP} and  read  Section~\ref{section:AKSTEST},
then come back to 
Section~\ref{section:EP} for technical details.

\section{Isogenies between elliptic curves}\label{section:EP}

In this section, we use isogenies between elliptic curves to construct
ring extensions.  To this end, we extend the methods introduced by
Couveignes and Lercier~\cite{CL} in two different directions.
Firstly, we provide efficient explicit expressions for the constants that
appear in the multiplication tensor of the ring of elliptic periods.
Thanks to these formulae, one can construct the ring of elliptic
periods in quasi-linear time.  Secondly, we explain how these methods,
originally 
introduced in the context 
of finite fields, can be  adapted   to the more general context of
rings.

We recall in Section~\ref{subsection:ECF} more or less classical
formulae about elliptic curves and isogenies over fields. In
Section~\ref{subsection:weierstrass}, these formulae are proved to
hold true over almost any base ring. In Section~\ref{subsection:ellp}, we use
isogenies to construct extension rings
and we finally
give a numerical example in Section~\ref{subsection:example}.
\medskip 

{\bf Notation}: If  $\valpha = (\alpha_i)_{i\in\ZZ/d\ZZ}$ and 
$\vbeta = (\beta_i)_{i\in\ZZ/d\ZZ}$ are two  vectors of length $d$,
we denote by  $\valpha \star_j \vbeta = \sum_i \alpha_i\beta_{j-i}$
the $j$-th component of the convolution product.
We denote by $\sigma(\valpha)=(\alpha_{i-1})_{i}$ the cyclic
shift of $\valpha$,
by $\valpha \diamond \vbeta
=(\alpha_i\beta_i)_i$ the component-wise
product and by
$\valpha \star \vbeta=(\valpha \star_i \vbeta)_i$ the convolution product.

\subsection{Elliptic curves over fields}
\label{subsection:ECF}

In this section, $\bK$ is a field with characteristic $p$ 
and $E/\bK$ is an elliptic curve given by a  Weierstrass equation
\begin{displaymath}
  Y^2Z+a_1XYZ+a_3YZ^2=X^3+a_2X^2Z+a_4XZ^2+a_6Z^3\,.
\end{displaymath}
We set 
\begin{eqnarray*}
&&b_2 = a_1^2+4a_2\,,\ 
b_4 = a_1a_3+2a_4\,,\ 
b_6 = a_3^2+4a_6\,,\\
&&b_8 = a_1^2a_6+4a_2a_6-a_1a_3a_4+a_2a_3^2-a_4^2\,.
\end{eqnarray*}
We denote by $O=[0:1:0]$ the origin.

  Following V{\'e}lu~\cite{velucras,veluthese} and Couveignes and Lercier
\cite{CL}, we state a few identities related to a degree
$d$ separable isogeny with cyclic kernel
$I : E\rightarrow E'$. We exhibit in Section~\ref{subsubsection:ENB} a normal basis for
the field extension $\bK(E)/\bK(E')$ consisting of degree $2$ functions.
We study in Section~\ref{subsubsection:dual} the matrix of the
trace form in this normal basis.

\subsubsection{Some simple elliptic functions}

If $A$ is a point in $E(\bbK)$,
we denote by $\tau_A : E\rightarrow E$ the translation by $A$.
Following~\cite[Section 2]{CL}, we set $x_A=x\circ \tau_{-A}$
and $y_A=y\circ \tau_{-A}$.

We check that
\begin{eqnarray}\label{eq:xA}
x_A\times ( x-x(A))^2&=&\left(a_3+2y(A)+a_1x(A)\right) y + x(A) x^2 +\nonumber \\
&+& \left( a_4+a_1^2x(A)+a_1a_3+2a_2x(A)+a_1y(A)+x(A)^2\right) x\nonumber
  \\
&+&a_3^2+a_1a_3x(A)+a_3y(A)+a_4x(A)+2a_6\,.
\end{eqnarray}
We do not give an explicit expression for $y_A$ but we check that
$y_A\times (x-x(A))^3$ can be written as a polynomial in
$\ZZ[a_1,a_2,a_3,a_4,a_6,x(A),y(A),x,y]$\,.
We also check that 
\begin{equation}\label{eq:xAxmA}
(x_A-x(A))(x_{-A}-x(A))=-\frac{\psi_3(a_1,a_2,a_3,a_4,a_6,x(A))}{(x-x(A))^2}
-\frac{\hat \psi_3(a_1,a_2,a_3,a_4,a_6,x(A))}{x-x(A)}
\end{equation}
where $\psi_3(a_1,a_2,a_3,a_4,a_6,x)$ is the so called $3$-division polynomial:
\begin{equation*}
\psi_3=3x^4+b_2x^3+3b_4x^2+3b_6x+b_8\,,
\end{equation*}
and 
\begin{equation*}
\hat\psi_3=\psi_3'/3=4x^3+b_2x^2+2b_4x+b_6\,.
\end{equation*}
We also check that  the resultant of $\psi_3$ and $\hat \psi_3$ in the variable $x$ is 
\begin{equation}\label{eq:resu23}
\Res_x(\psi_3,\hat \psi_3)=-\Delta^2
\end{equation}
where $\Delta \in \ZZ[a_1,a_2,a_3,a_4,a_6]$ is the discriminant
of the elliptic curve $E$.

If $A$, $B$ and $C$ are
 three pairwise  distinct points in $E(\bbK)$,
we define  $\Gamma(A,B,C)$ as in~\cite[Section 2]{CL},
\begin{equation}\label{eq:Gamma}
  \Gamma(A,B,C)=\frac{y(C-A)-y(A-B)}{x(C-A)-x(A-B)}\,.
\end{equation}
Taking for $C$ the generic point on $E$,  we define
a function $u_{A,B}\in \bbK(E)$ by $u_{A,B}(C)=\Gamma(A,B,C)$.  
It has two simple poles: one at $A$ and one at $B$.
The  following identities
are proven in~\cite[Section 2]{CL}.
\begin{eqnarray}
  \Gamma(A,B,C)&=&\Gamma(B,C,A)=-\Gamma(B,A,C)-a_1\,,\nonumber \\
&=&-\Gamma(-A,-B,-C)-a_1\,,\nonumber\\
  u_{A,B}+u_{B,C}+u_{C,A}&=&\Gamma(A,B,C)-a_1\,,\nonumber\\
  u_{A,B}u_{A,C}&=&x_A+\Gamma(A,B,C)u_{A,C}+\Gamma(A,C,B)u_{A,B} \nonumber \\
&&+a_2
  +x_{A}(B)+x_{A}(C)\,,\label{eq:uABuAC}\\
  u_{A,B}^2&=&x_A+x_B-a_1u_{A,B}+x_A(B)+a_2\,.\label{eq:uAB2}
\end{eqnarray}
We further can prove in the same way
\begin{eqnarray*}
  x_Cu_{A,B}&=&\Gamma(A,B,C)x_C+x_B(C)u_{C,B}-x_A(C)u_{C,A}+y_A(C)-y_B(C)\,,\\
  x_Au_{A,B}&=&y_A+x_B(A)u_{A,B}-y_B(A)\,,\\
  x_Bu_{A,B}&=&-y_B-a_1x_B-a_3+x_B(A)u_{A,B}-y_B(A)\,.
\end{eqnarray*}

\subsubsection{V{\'e}lu's formulae}\label{subsubsection:VF}

Let $d\geqslant 3$ be an odd integer and let $T\in E(\bK)$ be a point of order
$d$. 
For $k$ an integer, we set $x_k=x_{kT}$, $y_k=y_{kT}$ and
following V{\'e}lu~\cite{velucras}, we define
\begin{equation}\label{eq:x'y'}
x'=x+\sum_{1\leqslant k\leqslant d-1} \left[ x_{k}-x(kT)\right]\text{ and } y'=y+
\sum_{1\leqslant k\leqslant d-1} \left[y_{k}-y(kT)\right]\,.
\end{equation}
We also set
\begin{eqnarray*}
w_4&=&\sum_{1\leqslant k\leqslant (d-1)/2}6\,x(kT)^2+b_2\,x(kT)+b_4\,,\\
w_6&=&\sum_{1\leqslant k\leqslant (d-1)/2}10\,x(kT)^3+2\,b_2\,x(kT)^2+3\,b_4\,x(kT)+b_6\,,\\
a'_4&=&a_4-5w_4\,,\\
a'_6&=&a_6-b_2w_4-7w_6\,,
\end{eqnarray*}
and
\begin{equation}\label{eq:coeff'}
a'_1=a_1,\ a'_2=a_2,\ a'_3=a_3\,.
\end{equation}

V{\'e}lu proves the identity 
\begin{equation*}
(y')^2+a_1'x'y'+a_3'y'=(x')^3+a_2'(x')^2+a_4'x'+a_6'\,.
\end{equation*}
So the map $(x,y)\mapsto (x',y')$ defines a degree $d$
isogeny 
$I : E\rightarrow E'$ where $E'$ is the elliptic curve given by the
above Weierstrass equation. \medskip

\subsubsection{Elliptic normal basis}\label{subsubsection:ENB}
 
Let
\begin{equation}\label{eq:ukdef}
U_k=u_{kT,(k+1)T}\text{ and }u_k=\agot u_{kT,(k+1)T} + \bgot
\end{equation}
where
$\agot\not = 0$ and $\bgot$ are scalars in $\bK$ chosen such that 
\begin{equation}\label{eq:sum1}
\sum_{k\in
  \ZZ/d\ZZ}u_k=1.
\end{equation}
Such scalars always exist by~\cite[Lemma 4]{CL}.
For $k$ and $l$ distinct and non-zero in $\ZZ/d\ZZ$, we set
\begin{equation}\label{eq:deuz}
\Gamma_{k,l}=\Gamma(O,kT,lT).
\end{equation}
Recall 
\begin{equation}\label{eq:uOkt}
u_{O,kT}=\frac{y-y(-kT)}{x-x(kT)}\, .
\end{equation}
We  check that
\begin{equation}\label{eq:Uk}
U_k=u_{kT,(k+1)T}=u_{O,(k+1)T}-u_{O,kT}+\Gamma_{k,k+1}\,.
\end{equation}

The  system $(u_k)_{k\in \ZZ/d\ZZ}$ is a 
basis  of $\bK(E)$ over $\bK(E')$. More precisely, we have the following
lemma, that generalizes  Lemma~5 of~\cite{CL}.
\begin{lemma}[A normal basis]\label{lemma:indep}
Let $E$ be an  elliptic curve over a field 
$\bK$. Let $T\in E(\bK)$ be a point of odd order $d\geqslant 3$ and
 $I : E\rightarrow E'$ be the  degree $d$ separable isogeny defined from $T$  
by V\'elu's formulae. Let $(u_k)_{k\in \ZZ/d\ZZ}$ be the functions
in $\bK(E)$ defined  above.
Then the  system $(u_k)_{k\in \ZZ/d\ZZ}$ is a $\bK(E')$-basis of 
$\bK(E)$.\smallskip

Moreover, let $\bL \supset \bK$ be an extension of
$\bK$ and let $A\in E'(\bL)$ be a non-zero point. Let 
$B\in E(\bbL)$ be a point on $E$ such that $I(B)=A$ and let 
\begin{displaymath}
  I^{(-1)}(A)=[B]+[B+T]+[B+2T]+\cdots +[B+(d-1)T]
\end{displaymath}
be  the fiber of $I$ above $A$. 
Then the three following conditions are equivalent:
\begin{itemize}[itemsep=0.5pt,parsep=0.5pt,partopsep=0.5pt,topsep=0.5pt]
\item[(i)] The
images of the $(u_k)_{k\in \ZZ/d\ZZ}$ in the residue ring at $I^{-1}(A)$ form a
$\bL$-basis
of it;
\item[(ii)] The matrix $(u_k(B+lT))_{k,l\in \ZZ/d\ZZ}$ is invertible;
\item[(iii)] The point $A$ is not in the kernel of the dual isogeny $I':E'\rightarrow E$.
\end{itemize}
\end{lemma}
\begin{proof} 
We preliminary base change $E$ and $E'$ to $\bL$
and observe that the $(u_k)_{k\in \ZZ/d\ZZ}$
are $\bL$-linearly independent
and form a basis of the linear space $\cL(I^{-1}(O'))$ where
 $O'$ is the origin on $E'$ and 
$I^{-1}(O')=[O]+[T]+[2T]+\cdots+[(d-1)T]$
is the kernel of $I$. Indeed, let 
$(\lambda_k)_{k\in \ZZ/d\ZZ}$ be scalars in $\bL$ such that 
$f=\sum_{k\in \ZZ/d\ZZ}\lambda_ku_k$ is the zero function.
Taylor expansions of $f$ at poles of $u_k$ (see~\cite[Section 2]{CL})
show that all $\lambda_k$ are equal.
Since the sum of the $u_k$ is $1$, we deduce that every 
$\lambda_k$ is zero. So the $(u_k)_{k\in \ZZ/d\ZZ}$
are $\bL$-independent. They form a basis of
$\cL(I^{-1}(O'))$ because $I^{-1}(O')$ is a degree $d$ divisor
(Riemann Roch theorem).\medskip

Now, let us prove the second part of the lemma.\smallskip

To prove  that $(i)$ and $(ii)$ are equivalent, we  notice that
a vector $(\lambda_k)_{k\in \ZZ/d\ZZ}$ is in the kernel
of the matrix  $(u_k(B+lT))_{k,l\in
  \ZZ/d\ZZ}$ if and only if $\sum_{k\in
  \ZZ/d\ZZ}\lambda_ku_k(B+lT)$ is zero for every $l\in \ZZ/d\ZZ$.
This is equivalent to the vanishing of the function $\sum_{k\in
  \ZZ/d\ZZ}\lambda_ku_k$ on the fiber $I^{-1}(A)$.
Incidentally, we notice that  the matrix $(u_k(B+lT))_{k,l\in
  \ZZ/d\ZZ}$ is circulant.
  
\smallskip

To show that $(iii)$ implies $(i)$,
let
$(\lambda_k)_{k\in \ZZ/d\ZZ}$ be scalars in $\bL$ such that 
$f=\sum_{k\in \ZZ/d\ZZ}\lambda_ku_k$ vanishes on the fiber $I^{(-1)}(A)$.
If the $\lambda_k$ are not all zero, then 
$f$ is non-zero, and its divisor is $I^{(-1)}(A)-I^{(-1)}(O')$.
 We deduce that $\sum_{k\in \ZZ/d\ZZ}
[B+kT]-[kT]$ is a principal divisor. Thus $\sum_{k\in \ZZ/d\ZZ}
(B+kT-kT)=dB=I'(A)=O$, the origin on $E$. So $A$ lies in the kernel of $I'$. \smallskip

Conversely,  if $A$ lies in the kernel of $I'$, then the divisor
$I^{(-1)}(A)-I^{(-1)}(O')$ is principal. Let $f$ be a non-zero function on $E$
such that $(f)=I^{(-1)}(A)-I^{(-1)}(O')$. Since $f$
lies in $\cL(I^{-1}(O'))$, there exists a non-zero vector 
$(\lambda_k)_{k\in \ZZ/d\ZZ}$  in $\bL^d$ such that 
$f=\sum_{k\in \ZZ/d\ZZ}\lambda_ku_k$.  But  $f$ vanishes
on the fiber $I^{(-1)}(A)$, by construction. So $(i)$ implies
$(iii)$.\medskip

To finally prove the first part of the lemma, it is now enough to take for $A$
the generic point of $E'/\bK$. The generic
point  is not in
the kernel of $I'$ and thus the
system $(u_k)_{k\in \ZZ/d\ZZ}$ is a $\bK(E')$-basis of $\bK(E)$.
\end{proof}

\subsubsection{The trace form}\label{subsubsection:dual}

Lemma~\ref{lemma:indep} above provides a   basis for the residue ring
at a fiber $I^{-1}(A)=[B]+\ldots+[B+(d-1)T]$ where
 $A \in E'(\bK)$. We need 
fast algorithms for multiplying two elements in this residue ring, given
by their coordinates in our  basis.
 A prerequisite is
to determine the coordinates of  $x(B)$ in the basis
$(u_k(B))_{k\in \ZZ/d\ZZ}$. 
More generally, we are interested in 
the coordinates of  $x$ in the basis
$(u_k)_{k\in \ZZ/d\ZZ}$ of the $\bK(E')$-vector space $\bK(E)$. 
The reason is that when multiplying $u_k$ and $u_l$ there appear
some translates of $x$.  See Eqs.~(\ref{eq:uABuAC}) and (\ref{eq:uAB2}).
We will give explicit expressions for these coordinates
and explain  how to compute them efficiently. 
We shall make use of the  trace form of $\bK(E)/\bK(E')$. Remind
this is a  non-degenerate quadratic form.
For $f$ a function on $E$, we denote by $\Tr (f)$ the sum
$\sum_{k\in \ZZ/d\ZZ}f\circ \tau_{kT}$. It  can be seen as a function 
on $E'$.
Our goal is to compute 
$\Tr(u_{O,kT})$, $\Tr(u_ku_l)$ and 
$\Tr(u_kx)$ as linear combinations of $1$, $x'$ and $y'$. We then deduce
an explicit formula for the determinant of the trace form.

\paragraph{Traces of $u_{O,kT}$}\label{paragraph:tru0kt}

For $1\leqslant k\leqslant d-1$, we set 
\begin{math}
c_k=\Tr(u_{O,kT})\,.
\end{math}
It is  proven in~\cite[Section 4.2]{CL} that
\begin{equation}\label{eq:cgot1}
c_1=  \Tr(u_{O,T}) = \sum_{1\leqslant l\leqslant d-2}\Gamma_{l,l+1} - a_1\,.
\end{equation}
Assume $k$, $l$ and $k+l$ are non-zero in $\ZZ/d\ZZ$,
then
\begin{math}
\Tr(u_{O,(k+l)T})=\Tr(u_{O,kT})+\Tr(u_{O,lT})-d\Gamma_{k,k+l}\,.
\end{math}
Thus,
\begin{equation}\label{eq:recucgot}
c_{k+l}=c_k+c_l-d\Gamma_{k,k+l}\,.
\end{equation}

This formula enables us to compute all the $c_k$ for $1\leqslant k\leqslant
d-1$, at the expense of $O(d)$ operations in $\bK$.  Indeed, we first compute
the coordinates $(x(kT),y(kT))$ for $1\leqslant k\leqslant d-1$.  Then, using
Eqs.~(\ref{eq:Gamma}) and (\ref{eq:deuz}), we compute $\Gamma_{k,k+1}$ for
every $1\leqslant k\leqslant d-2$. We then use Eq.~(\ref{eq:cgot1}) to compute
$c_1$. Finally, we use Eq.~(\ref{eq:recucgot}) repeatedly for $l=1$ and
$1\leqslant k\leqslant d-2$, and we deduce the values of $c_2$, \ldots,
$c_{d-1}$.

\paragraph{Traces of $u_{k}u_l$}

Assume first   that $k\not \in \{-1,0,1\}$, so
$O$, $T$, $kT$ and $(k+1)T$ are pairwise distinct. Then 
\begin{eqnarray*}
U_0U_k&=&
u_{O,T}(u_{O,(k+1)T}-u_{O,kT}+\Gamma_{k,k+1})\,,\\
&=&x+\Gamma_{1,k+1}u_{O,(k+1)T}-\Gamma_{1,k+1}u_{O,T}+x(T)+x((k+1)T)\\
&&-x-\Gamma_{1,k}u_{O,kT}+\Gamma_{1,k}u_{O,T}-x(T)-x(kT)+\Gamma_{k,k+1}u_{O,T}\,,\\
&=& \Gamma_{1,k+1}(u_{O,(k+1)T}-u_{O,T})-\Gamma_{1,k}(u_{O,kT}-u_{O,T})\\
&&+x((k+1)T)-x(kT)+\Gamma_{k,k+1}u_{O,T}\,.
\end{eqnarray*}
So 
\begin{equation}\label{eq:tr0k}
\Tr(U_0U_k)=\Gamma_{1,k+1}(c_{k+1}-c_1)-\Gamma_{1,k}(c_k-c_1)
+d(x((k+1)T)-x(kT))+\Gamma_{k,k+1}c_1\,.
\end{equation}
For $k=0$, we have  $U_0^2=x+x_T-a_1u_{0,T}+x(T)+a_2$. And thus
\begin{equation}\label{eq:tr00}
\Tr(U_0^2)= 2x'+d(x(T)+a_2)-a_1c_1+2\sum_{1\leqslant l\leqslant d-1}x(lT)\,.
\end{equation}
For $k=-1$, we have 
\begin{eqnarray*}
U_0U_{-1}&=&u_{O,T}u_{-T,O}=-u_{O,T}u_{O,-T}-a_1u_{O,T}\,,\\
&=&-(x+\Gamma_{1,-1}u_{O,-T}-\Gamma_{1,-1}u_{O,T}+a_2+x(T)+x(-T))\,,\\
&=&-x+\Gamma_{1,-1}(u_{-T,O}+a_1)+\Gamma_{1,-1}u_{O,T}-a_2-2x(T)\,.
\end{eqnarray*}
And thus
\begin{equation}\label{eq:tr0m1}
\Tr(U_0U_{-1})=-x'+2\Gamma_{1,-1}c_1+d(a_1\Gamma_{1,-1}-a_2)-2dx(T)-\sum_{1\leqslant l\leqslant d-1}x(lT)\,.
\end{equation}
Finally, for $k=1$, we  have
\begin{equation}\label{eq:tr01}
\Tr(U_0U_1)=\Tr(U_{-1}U_{0})=\Tr(U_0U_{-1})\,.
\end{equation}
\medskip 

Now, for any $k$ and $l$, we have 
\begin{equation}\label{eq:scale}
\Tr(u_ku_l)=\agot^2\Tr(U_kU_l)+\bgot^2d+2\agot\bgot c_1\,.
\end{equation}
We set 
\begin{equation}\label{eq:defek}
\egot_k=\Tr(u_0u_k)\,.
\end{equation}
 This is a polynomial in $x'$ with degree one  if $k\in \{-1,0,1\}$,  and zero
otherwise.
We denote by $\vegot$ the vector $(\egot_k)_{k\in \ZZ/d\ZZ}$.

Assume now  we are given
a non-zero point $A\in E'(\bK)$.   
For every $k$ in $\ZZ/d\ZZ$, we write 
\begin{equation}\label{eq:egot}
  e_k=\egot_k(A)\,.
\end{equation}
We can compute the vector $\ve=(e_k)_{k\in \ZZ/d\ZZ}$ at the expense of $O(d)$
operations in $\bK$. We first compute 
the coordinates $(x(kT),y(kT))$ for $1\leqslant k\leqslant d-1$, the coefficients 
$\Gamma_{k,k+1}$ for every $1\leqslant k\leqslant d-2$ and the $c_k$ for $1\leqslant k\leqslant
d-1$ as explained in Section~\ref{paragraph:tru0kt}. We then compute the
$\Gamma_{1,k}$ for $2\leqslant k\leqslant d-1$ 
using Eqs.~(\ref{eq:Gamma}) and    (\ref{eq:deuz}).
Then, we use Eqs.~(\ref{eq:tr0k}), (\ref{eq:tr00}), (\ref{eq:tr0m1}),
and
(\ref{eq:tr01}) to compute the values of the $\Tr(U_0U_k)$  at $A$.  Finally,
we use Eq.~(\ref{eq:scale}) to deduce $\ve$.

\paragraph{Traces of $xu_{k}$}\label{paragraph:trxuk}

For $k\not \in \{-1,0\}$, we have 
\begin{eqnarray*}
xU_k&=&x_Ou_{kT,(k+1)T}\,,\\
&=&\Gamma_{k,k+1}x+x((k+1)T)u_{O,(k+1)T}-x(kT)u_{O,kT}+\\
&&y((k+1)T)-y(kT) + a_1(x((k+1)T)-x(kT))\,.
\end{eqnarray*}
And thus,
\begin{eqnarray}\label{eq:iotak}
\Tr(xU_k)&=&\Gamma_{k,k+1}(x'+\sum_{1\leqslant l\leqslant d-1}
x(lT))+x((k+1)T)c_{k+1}-x(kT)c_k+\nonumber \\
&&d(y((k+1)T)-y(kT) + a_1(x((k+1)T)-x(kT)))\,.
\end{eqnarray}
For $k=0$, we have 
\begin{equation*}
xU_0=x_Ou_{O,T}=y+x(T)u_{O,T}+y(T)+a_1x(T)+a_3\,.
\end{equation*}
And thus,
\begin{equation}\label{eq:iota0}
\Tr(xU_0)=y'+x(T)c_1+d(y(T)+a_1x(T)+a_3)+\sum_{1\leqslant l\leqslant d-1}y(lT)\,.
\end{equation}
For $k=-1$, we have 
\begin{equation*}
xU_{-1}=x_Ou_{-T,O}=-y-a_1x+x(T)u_{-T,O}+y(T)+a_1x(T)\,.
\end{equation*}
And thus,
\begin{equation}\label{eq:iota-1}
\Tr(xU_{-1})=-y'-a_1x' + x(T)c_1 + d(y(T)+a_1x(T))-\sum_{1\leqslant l\leqslant d-1}(y(lT)+a_1x(lT))\,.
\end{equation}

We set 
\begin{displaymath}
\igot_k=\Tr(xu_k)=\agot \Tr(xU_k)+\bgot (x'+\sum_{1\leqslant l\leqslant d-1}
x(lT))\,.
\end{displaymath}
This is a polynomial in $x'$ and $y'$ with total  degree at most $1$.
The vector $\vigot =(\igot_k)_{k\in \ZZ/d\ZZ}$ is
the  coordinate vector  of $x$ in the dual basis
of $(u_k)_{k\in \ZZ/d\ZZ}$.
Remind we are interested in the coordinates of $x$
in the basis $(u_k)_{k\in \ZZ/d\ZZ}$ itself. Call
$\vhigot = (\higot_k)_{k\in \ZZ/d\ZZ}$ these coordinates. We have
\begin{equation}\label{eq:star}
\vigot  = \vegot \star \vhigot.
\end{equation}

Assume now  we are given
a non-zero point $A\in E'(\bK)$.   
For every $k$ in $\ZZ/d\ZZ$, we write 

\begin{equation*}
\iota_k=\igot_k(A) \text{ and } \hiota_k=\higot_k(A).
\end{equation*}

We can compute the vector
 $\viota = (\iota_k)_{k\in \ZZ/d\ZZ}$ at the expense of $O(d)$
operations in $\bK$. 
Then, using Eq.~(\ref{eq:star}), we can compute
the vector $\vhiota = (\hiota_k)_{k\in \ZZ/d\ZZ}$
at the expense of one division in the degree $d$ convolution algebra
over $\bK$.
This boils down to $d(\log d)^2\log \log d$ operations in $\bK$.

\paragraph{The trace form}

We now study the trace form in the basis $(u_k)_{k\in \ZZ/d\ZZ}$.

The matrix $\left(\Tr(u_ku_l)\right)_{k,l}=\left(\egot_{l-k}  \right)_{k,l}$ 
is circulant and its determinant
is 
\begin{equation}\label{eq:Dinit}
D=\left|\Tr(u_ku_l)\right|_{k,l}=\prod_{k\in \ZZ/d\ZZ}\ \sum_{l\in \ZZ/d\ZZ}\zeta^{kl}\egot_l
\end{equation}
where $\zeta$ is a primitive $d$-th root of unity (that is $\zeta^d=1$ and 
$\zeta^k-1$ is a unit for every $1\leqslant k\leqslant d-1$).

We compute
\begin{displaymath}
  \sum_{l\in \ZZ/d\ZZ}\egot_l=\sum_{l\in \ZZ/d\ZZ} \Tr(u_0u_l)=
  \Tr(u_0\sum_{l\in \ZZ/d\ZZ} u_l)=\Tr(u_0)=1\,.
\end{displaymath}
Using Eqs.~(\ref{eq:tr0k}), (\ref{eq:tr00}), (\ref{eq:tr0m1}) and (\ref{eq:tr01}), we deduce that $D$ is
a degree $\leqslant d-1$ polynomial in $x'$ and the  coefficient of $(x')^{d-1}$ is 
\begin{displaymath}
  \agot^{2d-2}\prod_{1\leqslant k\leqslant d-1}(2-\zeta^k-\zeta^{-k})=\agot^{2d-2} d^2\,.
\end{displaymath}
Since $\egot_k=\egot_{-k}$ for every $k\in \ZZ/d\ZZ$, we deduce from
Eq.~(\ref{eq:Dinit}) that $D$ is a square.

We now assume that $d$ and the characteristic of $\bK$ are coprime. So
the degree
of $D(x')$ is $d-1$.
From Lemma~\ref{lemma:indep}, we deduce
that the roots of $D$ are the abscissae of points in the kernel of the
dual isogeny $I' :  E'\rightarrow E$  and they all have multiplicity two.
Using Eq.~(\ref{eq:x'y'}), we deduce
\begin{equation}\label{eq:Dpsi}
\psi_I^{2d}(x)D(x')=\agot^{2d-2}\psi_d^2(x)\,,
\end{equation}
where
\begin{equation}\label{eq:psiI}
\psi_I(x)=\prod_{1\leqslant k\leqslant (d-1)/2}(x-x(kT))
\end{equation}
is the factor of $\psi_d(x)$ corresponding to 
points in the kernel of $I$.

\paragraph{Example}

We detail on a simple example how to construct
a ring of  elliptic periods.
Following~\cite{CL}, we consider the
elliptic curve $E$ of order $10$ defined by
\begin{displaymath}
  E/\FF_{7} : {y}^{2}+xy+5\,y={x}^{3}+3\,{x}^{2}+3\,x+2\,.
\end{displaymath}
The point $T=(3,1)$ generates a subgroup $T\subset E(\FF_{7})$ of order $d=5$. The quotient
elliptic curve $E'=E/T$  given by V\'elu's formulae has equation
\begin{displaymath}
E'/\FF_{7} :  {y}^{2}+xy+5\,y={x}^{3}+3\,{x}^{2}+4\,x+6\,,
\end{displaymath}
and the quotient isogeny is 
\begin{multline*}
  I: (x,y) \mapsto (x',y')=\left({\frac
      {{x}^{5}+2\,{x}^{2}+5\,x+6}{{x}^{4}+3\,{x}^{2}+4}}, \right.\\
  \left.{\frac { \left( {x}^{6}+4\,{x}^{4}+3\,{x}^{3}+6\,{x}^{2}+3\,x+4 \right) y+3\,{
          x}^{5}+{x}^{4}+{x}^{3}+3\,{x}^{2}+4\,x+1}{{x}^{6}+{x}^{4}+5\,{x}^{2}+6}}\right)\,.
\end{multline*}
\smallskip

We focus first  on $\Tr(u_{O,t})$.
We have
\begin{displaymath}
  (u_{O,kt})_{1\leqslant k\leqslant d-1} = \left({\frac {y+2}{x+4}},{\frac
    {y+2}{x+3}},{\frac {y}{x+3}},{\frac {y+6}{x+4}}\right)\,.
\end{displaymath}
A direct but heavy calculation yields
\begin{displaymath}
  c_1=
       {\frac {y+2}{x+4}}+
       {\frac {y+2\,{x}^{2}+5}{{x}^{2}+5}}+
       {\frac {5}{x+3 }}+
         {\frac {6\,yx+3\,y+2\,{x}^{3}+3\,x}{ \left( {x}^{2}+5
             \right) \left( x+4 \right) }}+
         {\frac {6\,y+6\,x+4}{x+4}}=3\,.
\end{displaymath}
Alternatively, if we first compute
\begin{math}
  \Gamma_{1,2} = 2\,,\ \Gamma_{2,3} = 0\,,\ \Gamma_{3,4} = 2\,,
\end{math}
we more easily come to $c_1=2+0+2-1 = 3$.
From Eq.~(\ref{eq:recucgot}), we deduce
\begin{math}
  c_2=3\,,\ 
  c_3=6\,,\ 
  c_4=6.\,
\end{math}
\smallskip

Let us now consider
$\Tr(U_0^2)$. A direct calculation yields
\begin{eqnarray*}
  \Tr(U_0^2) &=&
  {\frac {(y+2)^2}{(x+4)^2}}+
  {\frac {(y+2\,{x}^{2}+5)^2}{({x}^{2}+5)^2}}+
  {\frac {5^2}{(x+3)^2}}+\\
  &&\hspace*{2cm}
  {\frac {(6\,y(x+3)+2\,{x}^{3}+3\,x)^2}{ \left( {x}^{2}+5
      \right)^2 \left( x+4 \right)^2 }}+
  {\frac {(6\,y+6\,x+4)^2}{(x+4)^2}}\,, \\
  &=&{\frac {2\,{x}^{5}+6\,{x}^{4}+{x}^{2}+3\,x+1}{{x}^{4}+3\,{x}^{2}+4}}\,.
\end{eqnarray*}
But we can easily deduce from Eq.~(\ref{eq:tr00}) that this is equal to
\begin{displaymath}
  2\,x'+  5\,(3+3)-1\,.\,3+2\,(3+ 4+ 4+ 3)\,.
\end{displaymath}
\smallskip

If we now look more carefully at  $\Tr(x\,U_0)$, we have
\begin{eqnarray*}
  \Tr(x\,U_0) &=&
  x\,.\,{\frac {y+2}{x+4}}+
   {\frac {3\,y+3\,{x}^{2}+4\,x+2}{{x}^{2}+x+2}}\,.\,
  {\frac {y+2\,{x}^{2}+5}{{x}^{2}+5}}+\\
  &&{\frac {2\,y+4\,{x}^{2}+3\,x+5}{{x}^{2}+6\,x+2}}\,.\,{\frac {5}{x+3 }}+\\
  &&
  {\frac {5\,y(x+1)+4\,{x}^{3}+6\,{x}^{2}+5\,x+6}{ \left( {x}^{2}+6\,x+2
      \right)  \left( x+3 \right) }}\,.\,{\frac {6\,yx+3\,y+2\,{x}^{3}+3\,x}{ \left( {x}^{2}+5
      \right) \left( x+4 \right) }}+\\
  &&{\frac {4\,y+3\,{x}^{2}+x+1}{{x}^{2}+x+2}}\,.\,{\frac
    {6\,y+6\,x+4}{x+4}}\,,\\
  &=& {\frac { y\,\left( {x}^{6}+4\,{x}^{4}+3\,{x}^{3}+6\,{x}^{2}+3\,x+4
 \right)+2\,{x}^{6}+3\,{x}^{
5}+3\,{x}^{4}+{x}^{3}+6\,{x}^{2}+4\,x+6}{{x}^{6}+{x}^{4}+5\,{x}^{2}+6}
}\,.
\end{eqnarray*}
But, from Eq.~(\ref{eq:iota0}), we find that this is equal to
\begin{displaymath}
  y'+3\,.\,3+5\,(1+1\,.\,3+5)+(1+ 0+ 5+ 5)\,.
\end{displaymath}
\smallskip

Let us finally notice that since $c_1=3\neq 0$, we can take
$\agot=1/c_1=3$ and $\bgot=0$\, (see
Section~\ref{subsubsection:ENB}). Moreover, let now $A=(4,2)\in
E'(\FF_7)$. Take $B\in E(\bFF_7)$ such that $I(B)=A$. We set
$\tau=x(B)\in \bFF_7$ and check that $\tau$ is a root of the
irreducible $\FF_7$-polynomial
\begin{math}
  ({x}^{5}+2\,{x}^{2}+5\,x+6)-4\,({x}^{4}+3\,{x}^{2}+4) =
  {x}^{5}+3\,{x}^{4}+4\,{x}^{2}+5\,x+4\,.
\end{math}
We find that
\begin{displaymath}
  \ve = (0, 4, 0, 0, 4)\,.
\end{displaymath}

\subsection{Universal Weierstrass elliptic curves}
\label{subsection:weierstrass}

All identities stated in Section~\ref{subsection:ECF} still make
sense and hold true for an elliptic curve
over a  commutative ring under some mild  restrictions. Some (but not all)
of these identities are
proven in this general context in  V{\'e}lu's thesis~\cite{veluthese}
and Katz and Mazur's book~\cite[Chapter 2]{K}.  In this section, we 
give an elementary proof for all the required identities. 
We consider in Section~\ref{sec:ring-elliptic-functions} a sort of universal  ring for Weierstrass curves with
torsion. This ring being an integral  domain, the identities hold true in its
fraction field. There only  remains to check the integrality 
of all quantities involved. By inverting the determinant 
of Eq.~(\ref{eq:Dinit}), we define in Section~\ref{sec:ring-elliptic-basis} a localization of the universal ring
where the system $(u_k)_{k\in \ZZ/d\ZZ}$ remains a basis for the function ring
extension associated to the isogeny.

\subsubsection{Division polynomials}\label{sec:torsion_points}

Let $A_1$, $A_2$, $A_3$, $A_4$ and $A_6$ be indeterminates and set
$B_2=A_1^2+4A_2$, $B_4=2A_4+A_1A_3$, $B_6=A_3^2+4A_6$,
$B_8=A_1^2A_6+4A_2A_6-A_1A_3A_4+A_2A_3^2-A_4^2$,  and 
\begin{displaymath}
  \Delta=-B_2^2B_8-8B_4^3-27B_6^2+9B_2B_4B_6\,.
\end{displaymath}
 Set 
\begin{displaymath}
\cA_1=\ZZ[A_1,A_2,A_3,A_4,A_6,\frac{1}{\Delta}]\,.
\end{displaymath}
Let $x$ and $y$ be two more indeterminates.
Set 
\begin{equation*}
\Lambda(A_1,A_2,A_3,A_4,A_6,x,y)=y^2+A_1xy+A_3y-x^3-A_2x^2-A_4x-A_6\in \cA_1[x,y]\,.
\end{equation*}
Let  $E_{\text{aff}}$ be the affine smooth plane  curve over $\cA_1$ with equation
$\Lambda (A_1,A_2,A_3,A_4,A_6,x,y)=0$. Let $E$ be the projective scheme over
$\cA_1$ with equation   $Y^2Z+A_1XYZ+A_3YZ^2=X^3+A_2X^2Z+A_4XZ^2+A_6Z^3$.
We denote by $O$ the
section $[0,1,0]$. We have $E_{\text{aff}}=E-O$ and $E$ is an elliptic curve
over (the spectrum of) $\cA_1$ in the sense of~\cite{K}.

For every integer $k\geqslant 0$, we denote by $\psi_k(A_1,A_2,A_3,A_4,A_6,x,y)$ the 
functions in $\cA_1[x,y]/(\Lambda)$ defined  recursively
as in~\cite[Prop. 3.53]{EngeBook}:
\begin{eqnarray*}
&&\psi_0=0,\  \psi_1=1,\  \psi_2=2y+A_1x+A_3,\,\\
&&\psi_3 = 3x^4+B_2x^3+3B_4x^2+3B_6x+B_8\,,\\
&&\psi_4 = \psi_2 \left(
2x^6+B_2x^5+5B_4x^4+10B_6x^3+10B_8x^2+\right.\\
&&\ \ \ \ \ \ \ \ \ \ \ \ \ \ \left.(B_2B_8-B_4B_6)x+B_4B_8-B_6^6 \right)\,,\\
&&\psi_{2k} =  \frac{\psi_k}{\psi_2}\left( \psi_{k+2}\psi_{k-1}^2 -
\psi_{k-2}\psi_{k+1}^2 \right)\,,\\
&&\psi_{2k+1} =   \psi_{k+2}\psi_{k}^3 -
\psi_{k-1}\psi_{k+1}^3\,.
\end{eqnarray*}

These are in $\cA_1[x,y]/(\Lambda)$ but we can see them
as polynomials in $\cA_1[x,y]$ with degree $0$ or $1$
in $y$. If $k$ is odd, then  $\psi_k$ belongs to
$\cA_1[x]$ and, as a polynomial
in $x$,  we have $\psi_k=kx^{\frac{k^2-1}{2}} +O(x^{\frac{k^2-3}{2}})$.
If $k$ is even, then $\psi_k/\psi_2$ belongs
to $\cA_1[x]$.
The ring  $\cA_1[x,y]/(\Lambda)$ is an integral  domain.
Following~\cite[Prop. 3.52, Prop. 3.55]{EngeBook}, we define the following elements of its field of fractions:
\begin{eqnarray*}
g_k&=&x-\frac{\psi_{k+1}\psi_{k-1}}{\psi_k^2}\,,\\
h_k&=&y+\frac{\psi_{k+2}\psi_{k-1}^2}{\psi_2\psi_k^3} +\left( 3x^2+2A_2x+A_4-A_1y\right)\frac{\psi_{k-1}\psi_{k+1}}{\psi_2\psi_{k}^2}\,.
\end{eqnarray*}
The following important relation holds true:
\begin{equation}\label{eq:diff}
g_k-g_l=-\frac{\psi_{k+l}\psi_{k-l}}{\psi_k^2\psi_l^2}\ \text{ if }k>l\geqslant 1\,.
\end{equation}
We recall 
that 
 multiplication by $k$  on $E-E[k]$ is given by $(x,y)
\mapsto (g_k,h_k)$. Indeed, this is well known on the generic fiber of
$E$  and it extends to all $E$ by (Zariski) continuity.

\subsubsection{Universal V\'elu's isogenies}\label{sec:ring-elliptic-functions}

Let $d\geqslant 3$ be an odd integer and let ``$x(T)$'' and
 ``$y(T)$'' be two more
indeterminates. Let $\cS$
 be the multiplicative 
subset in $\cA_1[x(T),y(T)]$
generated by all $\psi_k(x(T),y(T))$ for $1\leqslant k\leqslant d-1$.
Let $\cA_d$ be the ring
\begin{displaymath}
  \cA_d=\cA_1[x(T),y(T),\frac{1}{\cS},\frac{1}{d}]/(\psi_d(x(T)),\Lambda(A_1,A_2,A_3,A_4,A_6,x(T),y(T)))\,.
\end{displaymath}

This is an \'etale algebra over
$\cA_1[{1}/{d},{1}/{\cS}]$. Since the later is a regular ring,
$\cA_d$ is regular too. This is also
an integral  domain.  Indeed, the $d$-torsion of the generic
Weierstrass curve is irreducible. We denote by $\cK_d$ the field of fractions
of $\cA_d$.
The point $T=(x(T),y(T))$ defines a section of $E_{\text{aff}}$ over
$\cA_d$. 
The curve $E$, base changed to $\cA_d$,
may be seen as  the universal Weierstrass elliptic curve with a
point of exact order $d$ over a ring  where $d$ is invertible.

For every integer $k$ such that $1\leqslant k\leqslant d-1$, the point
$kT$ defines a section of $E$ over $\cA_{d}$. We call
$x(kT)$ and $y(kT)$ its coordinates and we have
\begin{eqnarray*}
  x(kT)&=&g_k(A_1,A_2,A_3,A_4,A_6,x(T),y(T))\in \cA_d\,,\\
  y(kT)&=&h_k(A_1,A_2,A_3,A_4,A_6,x(T),y(T))\in \cA_d\,.
\end{eqnarray*}

We note that due to Eq.~(\ref{eq:diff}), the difference
$x(lT)-x(kT)$ is a {\it unit} in $\cA_d$ for any 
 $k$ and $l$  in $\ZZ/d\ZZ$ such that $k$, $l$, $k+l$
and $k-l$ are not zero.  
If we base change
$E$ to $\cK_d$, we obtain an elliptic curve over a field and we 
can introduce all the scalars and functions of 
Section~\ref{subsection:ECF}: the $\Gamma_{k,l}$, the $x_k$, $y_k$,
$U_k$, $x'$, $y'$, $w_4$, $w_6$,  $c_k$\ldots 
The denominators arising in the definition
of these scalars and functions are units
in 
\begin{displaymath}
  \cA_d\left[E-E[d]\right]= \cA_d[\frac{1}{\psi_d(x)},
  x,y]/(\Lambda(A_1,A_2,A_3,A_4,A_6,x,y))\,.
\end{displaymath}
So all these scalars (resp. functions) are in $\cA_d$ (resp.
$\cA_d\left[E-E[d]\right]$). Especially, we can now define the
isogenous curve $E'$ thanks to Eq.~(\ref{eq:coeff'}), then the
isogenies $I$ and $I'$.

There remains to choose $\agot$ and $\bgot$. 
We just take $\agot=1$
and $\bgot=(1-c_1)/{d}$.  Then the functions 
$u_k=\agot U_k+\bgot$ are in $\cA_d\left[E-E[d]\right]$.
All equations from Eq.~(\ref{eq:deuz}) to
Eq.~(\ref{eq:psiI}) still hold true because they are true
in $\cK_d(E)$ and 
$\cA_d\left[E-E[d]\right]$ embeds in 
the later field.

\subsubsection{A normal  basis}\label{sec:ring-elliptic-basis}

The open subset $E-E[d]$ is the spectrum of the ring $\cA_d\left[ E-E[d] \right]$.
This is an integral domain and a  regular ring (because it is smooth over
$\cA_d$). Therefore it is integrally
closed.
The open subset $E'-\Ker I'$ is the spectrum of the ring 
\begin{displaymath}
\cA_d\left[E'-\Ker I' \right]=\cA_d[\frac{1}{D(x')},
  x',y']/(\Lambda(A_1',A_2',A_3',A_4',A_6',x',y'))\,.  
\end{displaymath}
This is again an integral domain and a  regular ring (because it is smooth over
$\cA_d$). Therefore it is integrally
closed too.
Eqs.~(\ref{eq:xA}), (\ref{eq:x'y'}), (\ref{eq:Dpsi}) and (\ref{eq:psiI}) show that 
$\cA_d\left[E'-\Ker I'\right]$ is included in $\cA_d\left[ E-E[d]\right]$.
Eqs.~(\ref{eq:xA}) and  (\ref{eq:x'y'}) prove that $x$ and $y$ are
integral over $\cA_d\left[E'-\Ker I'\right]$. We deduce that the translates $(x_k)_{1\leqslant
  k\leqslant d-1}$ and
$(y_k)_{1\leqslant k\leqslant d-1}$
are integral over $\cA_d\left[E'- \Ker I'\right]$ too. Using Eq.~(\ref{eq:xAxmA}), we
deduce that the ${1}/(x-x(kT))$ are integral over $\cA_d\left[E'-\Ker I'\right]$.
Note that in the special
case $d=3$,  we also
need Eq.~(\ref{eq:resu23}).
Now Eqs.~(\ref{eq:Dpsi}) and (\ref{eq:psiI}) prove that
${1}/{\psi_d(x)}$ is integral over $\cA_d\left[ E'-\Ker I'\right]$.
Altogether   $\cA_d\left[ E-E[d]\right]$ is the  integral closure
 of $\cA_d\left[E'-\Ker I'\right]$ in $\cK_d(E)$.

Using Eqs.~(\ref{eq:uOkt}) and (\ref{eq:Uk}) and the fact that 
the ${1}/(x-x(kT))$ are integral over $\cA_d[E'-$ $\Ker I']$, 
we show  that the
$(u_k)_{k\in \ZZ/d\ZZ}$
are integral over $\cA_d\left[E'-\Ker I'\right]$,   therefore  
belong to $\cA_d\left[E-E[d]\right]$.
For
every function $f$ in $\cA_d\left[E-E[d]\right]$, the products
$fu_k$ are  integral over $\cA_d\left[E'-\Ker I'\right]$. Therefore their traces $\Tr(fu_k)$ 
 belong to $\cA_d\left[E'-\Ker I'\right]$, since  this ring  is integrally closed. Now remember that 
the determinant of the trace form is 
\begin{equation*}
D(x')=\left| \Tr(u_ku_l) \right|_{k,l}\,,
\end{equation*}
a unit in $\cA_d\left[E'-\Ker I'\right]$. We deduce that the coordinates of $f$ in the basis $(u_k)_{k\in \ZZ/d\ZZ}$
are in $\cA_d\left[E'-\Ker I'\right]$.  We thus have found a basis for the $\cA_d\left[E'-\Ker I'\right]$-module 
$\cA_d\left[E-E[d]\right]$. This finite free module of rank $d$
is also \'etale because the determinant $D(x')$
of the trace form is a unit.

Let $\sigma$ be the $\cA_d\left[E'-\Ker I' \right]$-automorphism of
$\cA_d\left[ E-E[d]\right]$
induced
 by  the translation $\tau_{-T}$. We have $\sigma(u_k)=u_{k+1}$ for every 
 $k\in \ZZ/d\ZZ$.


\begin{lemma}[A freeness result]\label{lemma:EE'}
The ring 
\begin{displaymath}
  \cA_d\left[E-E[d]\right]=\cA_d[\frac{1}{\psi_d(x)}, x,y]/(\Lambda(A_1,A_2,A_3,A_4,A_6,x,y))
\end{displaymath}
is a finite free \'etale   algebra
of rank $d$ over 
\begin{displaymath}
  \cA_d\left[E'-\Ker I'\right]=\cA_d[ \frac{1}{D(x')},
x',y']/(\Lambda(A_1',A_2',A_3',A_4',A_6',x',y'))
\end{displaymath}
and 
 $(u_{k})_{1\leqslant k\leqslant d-1}$ is a  basis
for this free algebra.  For every 
 $k\in \ZZ/d\ZZ$, we have $\sigma(u_k)=u_{k+1}$ where
 $\sigma$ is  the $\cA_d\left[E'-\Ker I' \right]$-automorphism of
$\cA_d\left[ E-E[d]\right]$
induced
 by  the translation $\tau_{-T}$. 
\end{lemma}

The following theorem is proven by  base change in
 Lemma~\ref{lemma:EE'}.
\begin{theorem}[Elliptic Kummer extension]\label{th:formules}
Let $d\geqslant 3$ be an odd integer.
Let $R$ be a ring where $d$ is invertible.
 Let  $a_1$, $a_2$, $a_3$, $a_4$, $a_6$, $\xgot$
and $\ygot$ be elements in $R$  such that
\begin{itemize}[itemsep=0.5pt,parsep=0.5pt,partopsep=0.5pt,topsep=0.5pt]
\item $\Delta(a_1,a_2,a_3,a_4,a_6)$ is a unit in $R$,
\item $\psi_d(a_1,a_2,a_3,a_4,a_6,\xgot,\ygot)=0$,
\item $\psi_k(a_1,a_2,a_3,a_4,a_6,\xgot,\ygot)$
is a unit in $R$ for any $1\leqslant k\leqslant d-1$.
\end{itemize}
Then $T=(\xgot,\ygot)$ is a point of  exact order $d$ on the 
Weierstrass elliptic
curve given by the equation
$y^2+a_1xy+a_3y=x^3+a_2x^2+a_4x+a_6$  over $R$. \medskip

Set $\agot = 1$
and $\bgot=(1-c_1)/{d}$
and $u_k=\agot U_k+\bgot$. 
Then all equations from Eq.~(\ref{eq:deuz}) to
Eq.~(\ref{eq:psiI}) still make sense and 
hold true in the ring 
\begin{displaymath}
  R\left[ E-E[d]\right]=R[\frac{1}{\psi_d(x)},x,y]/(\Lambda(a_1,a_2,a_3,a_4,a_6,x,y))
\end{displaymath}
and this ring  is a finite free \'etale
algebra 
of rank $d$ over 
\begin{displaymath}
  R\left[E'-\Ker I' \right]=R[\frac{1}{D(x')},x',y']/(\Lambda(a_1',a_2',a_3',a_4',a_6',x',y'))
\end{displaymath}
and
$(u_{l})_{1\leqslant l\leqslant d-1}$ is a  basis
for this free algebra.  \medskip

For every 
 $k\in \ZZ/d\ZZ$, we have $\sigma(u_k)=u_{k+1}$ where
 $\sigma$ is  the $R\left[E'-\Ker I' \right]$-automorphism of 
  $R\left[ E-E[d]\right]$  induced  by  the translation $\tau_{-T}$.

\end{theorem}

\subsection{Rings of elliptic periods}\label{subsection:ellp}

In this section, we give a recipe for constructing  an extension of a ring $R$
using  an isogeny between two elliptic curves over $R$.
The resulting ring will
be called a  {\it ring of elliptic periods}. It will be a finite
free \'etale algebra
over $R$. We just adapt the construction of~\cite[Section 4]{CL}
to the case where the base ring is no longer   a field.
So in this section,  $R$ is  a ring and $d\geqslant 3$ is an odd integer.
We assume that 
$d$ is invertible  in $R$ and  that we are given
 an elliptic curve $E$  over $R$  by  its 
Weierstrass equation 
$y^2+a_1xy+a_3y=x^3+a_2x^2+a_4x+a_6$  where 
$\Delta(a_1,a_2,a_3,a_4,a_6)$ is a unit in $R$. We also are given a $R$-point
$T=(\xgot,\ygot)$ on $E$ with exact order $d$. We call $I : E\rightarrow E'$
the corresponding isogeny, given by V{\'e}lu's formulae.
Let    $D(x')=\left| \egot_{l-k} \right|_{k,l}$ be the  polynomial
in $R[x']$ defined by Eqs.~(\ref{eq:Dinit}),  (\ref{eq:Dpsi}) and (\ref{eq:defek}).

We further assume that we are given a section
$A=(x'(A),y'(A))\in E'(R)$ of $E'_{\text{aff}}\rightarrow \Spec(R)$. 
We assume that $D(x'(A))$ is a unit in $R$. 
Geometrically, this means that the section
$A$ does not intersect the kernel of the dual
isogeny $I':E'\rightarrow I$.
This is equivalent to
the circulant matrix $\left( \egot_{l-k}(A) \right)_{k,l}$ being
invertible.  For every $k$ in $\ZZ/d\ZZ$, we write 
\begin{math}
e_k=\egot_k(A).
\end{math}
This is an element of $R$.
Saying that  the circulant matrix $\left( e_{l-k}\right)_{k,l}$ is
invertible means that the  vector $\ve=(e_k)_{k\in \ZZ/d\ZZ}$ is invertible
for the convolution product $\star$ on $R^d$.
We denote by $\veinv$ the inverse of $\ve$ for the convolution product.
The ideal $(x'-x'(A),y'-y'(A))$
of $R\left[E-E[d]\right]= R[x,y,{1}/{\psi_d(x)}]/(\Lambda(a_1,a_2,a_3,a_4,a_6,x,y))$ is denoted $\Fgot_A$.
We call 
\begin{displaymath}
  S=R[x,y,\frac{1}{\psi_d(x)}]/(\Lambda(a_1,a_2,a_3,a_4,a_6,x,y),\Fgot_A)\,,
\end{displaymath}
the residue ring of 
$I^{-1}(A)$.  We say that $S$ is a ring of elliptic periods.   
If we specialize at $A$ in  Theorem~\ref{th:formules}, we find that
 $S$ is a finite free \'etale   $R$-algebra  with basis $\Theta=
(\theta_{k})_{k \in \ZZ/d\ZZ}$
where  
\begin{equation*}
\theta_k=u_k\bmod \Fgot_A.
\end{equation*}
 We call
 $\sigma : S\rightarrow S$ be the $R$-automorphism induced
on $S$   by  the translation $\tau_{-T}$,
\begin{displaymath}
  \begin{array}{crcl}
    \sigma :& S & \longrightarrow  &   S\,,\\
    &f\bmod \Fgot_A& \longmapsto  &f\circ\tau_{-T}\bmod \Fgot_A\,.
  \end{array}
\end{displaymath}
It is clear that $\sigma(\theta_k)=\theta_{k+1}$ for all
$k\in \ZZ/d\ZZ$. So, if $\alpha=\sum_{k\in \ZZ/d\ZZ}\alpha_k\theta_k$
is an element of $S$ with coordinates $\valpha=(\alpha_k)_{k\in \ZZ/d\ZZ}
\in R^d$ in the basis $\Theta$, then the  coordinate vector
 of $\sigma(\alpha)$ is 
the cyclic shift
$\sigma(\valpha)=(\alpha_{k-1})_{k\in \ZZ/d\ZZ}$ of 
$\valpha$. We see that  the $R$-automorphism $\sigma : S\rightarrow S$ 
of the free $R$-algebra $S$ takes a very simple
form on the basis $\Theta$.

We call
$\cL\subset R(E-E[d])$ the $R$-module generated
by the $u_k$ for $k \in \ZZ/d\ZZ$. 
We know  that reduction modulo $\Fgot_A$ defines an isomorphism of $R$-modules:
\begin{displaymath}
  \begin{array}{crcl}
    \epsilon_A : & \cL & \longrightarrow  &   S\,,\\
    &f & \longmapsto  &f\bmod \Fgot_A\,.
  \end{array}
\end{displaymath}
So elements in $S$ can be represented by elements in $\cL$.

We now study the multiplication tensor in $S$.
We shall find a simple expression for this tensor
using
interpolation at some auxiliary points, in the spirit of  discrete  Fourier transform.
We  first notice   that if $k, l \in
\ZZ/d\ZZ$ and $k\not = l,l+1,l-1 \bmod d$, then
\begin{displaymath}\label{eq:ukul}
u_ku_l\in \cL.
\end{displaymath} 

This is proven using Eqs.~(\ref{eq:uABuAC}),   (\ref{eq:ukdef}), and
(\ref{eq:Uk}).
Using Eqs.~(\ref{eq:uABuAC}), (\ref{eq:uAB2}),   (\ref{eq:ukdef}), and
(\ref{eq:Uk}), we also show that 
\begin{displaymath}\label{eq:uk2}
u_{k-1}u_{k} +\agot^2 x_k \in \cL
\text{ and }
u_{k}^2-\agot^2x_k-\agot^2x_{k+1}\in \cL\,.
\end{displaymath}

So if $(\alpha_k)_{k \in \ZZ/d\ZZ}$
and $(\beta_k)_{k \in \ZZ/d\ZZ}$ are two vectors in $R^d$, we have
\begin{eqnarray}\label{eq:mulprinc}
(\sum_{k} \alpha_ku_{k})(\sum_{k}
\beta_ku_{k})&=&\agot^2
\sum_k\alpha _k\beta_k(x_k+x_{k+1})-
\agot^2\sum_k \alpha_{k-1}\beta_{k}x_{k} - 
\agot^2\sum_k
\beta_{k-1}\alpha_{k}x_{k}\bmod \cL \nonumber \\
&=&\agot^2 \sum_k(\alpha_{k}-\alpha_{k-1})(\beta_k-\beta_{k-1})x_k\bmod \cL.
\end{eqnarray}

We now  assume we are given  an auxiliary  section $M=(x(M),y(M))$
of $E_{\text{aff}}\rightarrow \Spec(R)$ such that the image
$N=I(M)$ of $M$ by $I$ is a section $(x'(N),y'(N))$ of 
$E'_{\text{aff}}\rightarrow \Spec(R)$ and $D(x'(N))$ is a unit
in $R$. So, the residue ring at $I^{-1}(N)$ is a free $R$-module 
of rank $d$ and the evaluation map
\begin{displaymath}
  \begin{array}{crcl}
    \epsilon_N : & \cL & \longrightarrow  &   R^d\,,\\
    &f & \longmapsto  &(f(M+kT))_{k \in \ZZ/d\ZZ}\,.
  \end{array}
\end{displaymath}
is a bijection. Also, the vector 
\begin{equation}
  \label{eq:vun}
\vuN=(u_0(M+kT))_{k\in \ZZ/d\ZZ}
\end{equation}
is invertible  for the convolution product in $R^d$. We call
$\vuNinv$ its inverse. 
We denote by $\vxN$ the vector 
\begin{equation}
  \label{eq:vxn}
  \vxN=\epsilon_N(x)=(x(M+kT))_{k\in \ZZ/d\ZZ}\,.
\end{equation}

We note 
\begin{equation*}
\xi_k=x_k\bmod \Fgot_A
\end{equation*}
 for every $k\in \ZZ/d\ZZ$.
Since $S$ is free over $R$ and $\Theta$ is  a basis for it,
there exist
 scalars 
 $(\hiota_k)_k$ in $R$ such that
\begin{equation*}
\xi_0=\sum_{k\in \ZZ/d\ZZ}\hiota_k\theta_k.
\end{equation*}

So $\hviota=(\hiota_k)_k$
is the  coordinate vector of $\xi_0$ in the basis $\Theta$.
In Section~\ref{paragraph:trxuk},
we already explained how to compute these  coordinates in
quasi-linear time  in the dimension $d$. 

Let $\alpha$, $\beta$ and $\gamma$ be three elements in
$S$ such that $\gamma=\alpha\beta$. Let $\valpha=(\alpha_k)_{k\in \ZZ/d\ZZ}$
be the coordinate vector of $\alpha$ in the basis $\Theta$. Define 
$\vbeta$ and $\vgamma$ in a similar way. 
To compute the multiplication tensor, we use  an argument similar
to the one of~\cite[Section 4.3]{CL}.
We define four functions in  $\cA_d\left[E-E[d]\right]$,
\begin{eqnarray*}
f_\alpha&=&\sum_i \alpha_i u_i\,,\ f_\beta = \sum_i \beta_iu_i\,,\\
\cQ &=& \agot^2 \sum_i (\alpha_i-\alpha_{i-1})(\beta_i-\beta_{i-1})x_i\,,\\
\cR&=&f_\alpha f_\beta-\cQ\,.
\end{eqnarray*}
The product 
we want to  compute is $f_\alpha f_\beta=\cQ+\cR\bmod \Fgot_A$.
From Eq.~(\ref{eq:mulprinc}), we deduce
that $\cR$ is in $\cL$.
From  the definition of $\hviota$, we deduce that the
coordinates in $\Theta$ of $\cQ\bmod \Fgot_A$
are given by the vector
\begin{displaymath}
  \hviota\star \left( \agot^2 (\valpha -\sigma(\valpha))\diamond
    (\vbeta-\sigma(\vbeta))\right)\,.
\end{displaymath}
The evaluation of $f_\alpha$ at the points $(M+kT)_k$ is the vector
$\epsilon_N(f_\alpha)=\vuN \star \valpha$. The evaluation
 of $\cR$ is
$\epsilon_N(\cR)=(\vuN\star \valpha)\diamond (\vuN\star \vbeta)-\vxN\star 
(\agot^2(\valpha -\sigma(\valpha))\diamond (\vbeta-\sigma(\vbeta)))$.
If we $\star$ multiply this last  vector on the left
by $\vuNinv$, we obtain the coordinates of $\cR$
in the basis $(u_0,\ldots,u_{d-1})$. These are the coordinates
of $\cR\bmod \Fgot_A$ in the basis $\Theta$ too.

So  the multiplication tensor in the $R$-basis
$\Theta$ of the free $R$-algebra 
$S$   is given by 
\begin{multline} \label{eq:tensor}
\vgamma=
(\agot^2\hviota)\star \left( (\valpha -\sigma(\valpha))\diamond (\vbeta-\sigma(\vbeta))
\right)+\\
\vuNinv \star \left((\vuN\star \valpha)\diamond (\vuN\star \vbeta)-(\agot^2\vxN)\star 
\left((\valpha -\sigma(\valpha))\diamond (\vbeta-\sigma(\vbeta))\right)\right)
\end{multline}
This multiplication tensor   consists of $5$ convolution
products,  $2$ component-wise products, $1$ addition and $3$
subtractions between vectors in $R^d$.\medskip

The following theorem summarizes the results in this section.
\begin{theorem}[The ring of elliptic periods]\label{th:tensor}
Let $d\geqslant 3$ be an odd integer.
Let $R$ be a ring where $d$ is invertible.
 Let  $a_1$, $a_2$, $a_3$, $a_4$, $a_6$, $\xgot$
and $\ygot$ be elements in $R$  such that
$\Delta(a_1,a_2,a_3,a_4,a_6)$ is a unit in $R$ and
the point $T=(\xgot,\ygot)$ is a point of  exact order $d$ on the 
Weierstrass elliptic
curve   over $R$ given by the equation
$y^2+a_1xy+a_3y=x^3+a_2x^2+a_4x+a_6$. 
Let $I : E\rightarrow E'$ be the V{\'e}lu's isogeny with kernel
$\langle T\rangle$ and let $A=(x'(A),y'(A))\in E'(R)$ be a section
of  $E'_{\text{aff}}\rightarrow \Spec(R)$ that does not
intersect the kernel of the dual
isogeny $I':E'\rightarrow I$ (equivalently  $D(x'(A))$ is a unit in $R$).
Let $\Fgot_A = (x'-x'(A),y'-y'(A))$ be the corresponding ideal 
of $R\left[E-E[d]\right]= R[x,y,{1}/{\psi_d(x)}]/(\Lambda(a_1,a_2,a_3,a_4,a_6,x,y))$.
Let 
\begin{displaymath}
  S=R[x,y,{1}/{\psi_d(x)}]/(\Lambda(a_1,a_2,a_3,a_4,a_6,x,y),\Fgot_A)\,,
\end{displaymath}
be the residue ring of 
$I^{-1}(A)$.  
Then  $S$ is a finite free \'etale   $R$-algebra of rank $d$.
If we  call
 $\sigma : S\rightarrow S$ the $R$-automorphism induced
on $S$,   by  the translation $\tau_{-T}$, then
$S$ is a free $R[\sigma]$-module  of rank $1$.

Using notations introduced 
from Eq.~(\ref{eq:deuz}) to
Eq.~(\ref{eq:psiI}), we set $\agot = 1$,
$\bgot=(1-c_1)/{d}$,
$u_k=\agot U_k+\bgot$ and $\theta_k=u_k\bmod \Fgot_A$. Then 
$\sigma(\theta_k)=\theta_{k+1}$, and 
$\Theta=
(\theta_{k})_{k \in \ZZ/d\ZZ}$ is an $R$-basis of $S$.
If  $M=(x(M),y(M))\in E(R)$ is an auxiliary  section that does not
cross $E[d]$, then the multiplication tensor of $S$ in the basis $\Theta$
is given by Eq.
(\ref{eq:tensor}).

\end{theorem}

\subsection{Example}\label{subsection:example}

Let $R$ be the ring $\ZZ/101^2\ZZ$. We
 consider the elliptic curve $E$ over $R$
defined by the Weierstrass equation 
\begin{math}
  E/(\ZZ/101^2\ZZ) : {y}^{2}={x}^{3}+55\,x+91\,.
\end{math}
Let $T$ be the point $(659,8304)\in E/(\ZZ/101^2\ZZ)$. This is a point
with exact order  $d=7$.\smallskip

We first compute
\begin{math}
  \Gamma_{1,2} = 5780\,,\ \Gamma_{2,3} = 4390\,,\ \Gamma_{3,4} = 3596\,,\
  \Gamma_{4,5} = 4390\text{ and }\Gamma_{5,6} = 5780\,.
\end{math}
We then find $c_1=3534$, and
from Eq.~(\ref{eq:recucgot}), we deduce
\begin{math}
  c_2=7412\,,\ 
  c_3=618\,,\ 
  c_4=9583\,,\ 
  c_5=2789\text{ and } 
  c_6=6667\,.
\end{math}
Moreover $c_1$ is a unit in $R$ and we set
 $\agot=1/c_1=6665$ and $\bgot=0$.
\smallskip

We compute the quotient elliptic curve $E'=E/\langle T\rangle$ thanks to V\'elu's
formulae. This yields the curve
\begin{math}
  E'/(\ZZ/101^2\ZZ) : {y}^{2}={x}^{3}+6725\,x+6453\,.
\end{math}
Let  $A$ be the point $(1373,1956)\in E'(\ZZ/101^2\ZZ)$. This is
a point with exact  order 14.

We can efficiently compute traces of $u_{k}u_l$ evaluated at $A$ with
Eqs.~\eqref{eq:tr0k}, \eqref{eq:tr00}, \eqref{eq:tr0m1}, \eqref{eq:tr01} and
\eqref{eq:scale}. We find
\begin{displaymath}
  \ve = (9428,\,6046,\,1946,\,2596,\,2596,\,1946,\,6046)\,.
\end{displaymath}
This vector is invertible for the convolution product in $R^d$ and its
inverse is
\begin{displaymath}
  \veinv = (3392,\, 3344,\, 10161,\, 101,\, 101,\, 10161,\, 3344 )\,.
\end{displaymath}
We now  compute traces of $xu_{k}$ evaluated at $A$ with
Eqs.~\eqref{eq:iotak}, \eqref{eq:iota0}, \eqref{eq:iota-1}
and~\eqref{eq:star}, and find
\begin{displaymath}
  \viota  = ( 10063,\, 4509,\, 6660,\, 4259,\, 6660,\, 4509,\, 138)\,.
\end{displaymath}
We finally obtain
\begin{displaymath}
   \vhiota =  \veinv \star \viota = (7790,\, 6555,\, 2470,\, 2741,\, 4358,\, 2047,\, 636)\,.
\end{displaymath}
Let us consider the additional evaluation point $M=(8903,\, 4033)\in
E(\ZZ/101^2\ZZ)$. 
We check that $(\egot_k(N))_{k}$ where  $N=I(M)$ is invertible for the convolution
product in $R^d$. So $N$ does not cross the kernel of the dual
isogeny. Then Eq.~\eqref{eq:vxn} yields
\begin{displaymath}
   \agot^2 \vxN = (2742,\, 2044,\, 649,\, 2348,\, 7216,\, 9732,\, 7464)\,.
\end{displaymath}
Similarly, Eq.~\eqref{eq:vun} yields
\begin{displaymath}
  \vuN = (1029,\, 7201,\, 10176,\, 1807,\, 4875,\, 3261,\, 2255)\,.
\end{displaymath}
And therefore,
\begin{math}
  \vuNinv = (7790,\, 1761,\, 3889,\, 6998,\, 5866,\, 1090,\, 3210)\,.
\end{math}
\medskip

Now, let us make use of these precomputations to, for instance, compute
$\theta_0^2$ with Eq.~\eqref{eq:tensor}. We thus start from
\begin{math}
  \valpha = ( 1,\, 0,\, 0,\, 0,\, 0,\, 0,\, 0)\,,
\end{math}
and we first compute
\begin{displaymath}
  \vuN\star \valpha = (1029,\, 7201,\, 10176,\, 1807,\, 4875,\, 3261,\, 2255)\,,
\end{displaymath}
and
\begin{displaymath}
  \agot^2 \vxN \star \left((\valpha -\sigma(\valpha))\diamond (\valpha-\sigma(\valpha))\right)  =  ( 5,\, 4786,\, 2693,\, 2997,\, 9564,\, 6747,\, 6995 )\,.
\end{displaymath}
Thus,
\begin{multline*}
  \vuNinv \star \left((\vuN\star \valpha)\diamond (\vuN\star \vbeta)-(\agot^2\vxN)\star 
    \left((\valpha -\sigma(\valpha))\diamond (\vbeta-\sigma(\vbeta))\right)\right) =\\
  (8133,\, 8133,\, 8133,\, 8133,\, 8133,\, 8133,\, 8133)\,.
\end{multline*}
It follows,
\begin{displaymath}
  (\agot^2\hviota)\star \left( (\valpha -\sigma(\valpha))\diamond (\vbeta-\sigma(\vbeta))
\right)=
(6406,\, 4952,\, 8520,\, 969,\, 8109,\, 7516,\, 7834 )\,,
\end{displaymath}
and finally
\begin{displaymath}
  \vgamma = (4338,\, 2884,\, 6452,\, 9102,\, 6041,\, 5448,\, 5766)\,.
\end{displaymath}

\section{An elliptic AKS criterion}\label{section:AKSTEST}

Agrawal, Kayal and Saxena have proven~\cite{AKS} that primality
of an integer $n$  can be
tested in deterministic polynomial time $(\log n)^{\frac{21}{2}+o(1)}$. 
Their test, often called the AKS test, relies on explicit
computations  in the multiplicative group of a
well chosen free  commutative $R$-algebra $S$ of 
finite rank, where $R=\ZZ/n\ZZ$.
More precisely, they take for $S$ the cyclic algebra 
$R[x]/(x^r-1)$ where $r$ is a well chosen, and rather
large,  integer.

Lenstra 
and Pomerance generalized this algorithm and obtained the  better
deterministic complexity  $(\log n)^{6+o(1)}$~\cite{LP}. The main improvement
in Lenstra and Pomerance's approach consists in using 
a more general construction for  the free commutative
algebra $S$. As a consequence,
the dimension of $S$ is much smaller for a given $n$, and this results
in a faster algorithm.
A nice survey 
\cite{Schoof} has been written by Schoof.

Berrizbeitia first~\cite{Berri}, and  then  Cheng~\cite{Cheng},  
have  proven that there exists
a probabilistic variant of these algorithms that 
works in time $(\log n)^{4+o(1)}$ provided $n-1$ has a divisor
$d$ bigger than $(\log _2 n)^2$ and smaller than 
a constant times $(\log _2 n)^2$.  
Avanzi and  Mih\u{a}ilescu~\cite{AvM},
and independently Bernstein~\cite{B}, explain
how to treat a  general integer $n$ using a divisor $d$
 of $n^f-1$ instead, where $f$ is a small integer.
The initial idea consists in using
$R$-automorphisms of $S$  to speed up the 
calculations. In these variants, the free 
commutative $R$-algebra 
$S$ has to be constructed in such a way that a non-trivial
$R$-automorphism $\sigma : S\rightarrow S$ is effectively 
given, and 
can be efficiently applied to any element in $S$.

All the aforementioned  algorithms  construct 
$S$ as  a residue ring modulo $n$
of a cyclotomic or  Kummer extension of
the ring  $\ZZ$ of integers. 
In this section,  we propose an
AKS-like  primality criterion
that relies on Kummer theory of elliptic curves.
The  main advantage of this elliptic variant,
compared to  the Berrizbeitia-Cheng-Avanzi-Mihailescu-Bernstein one,
is
that it  allows  a much greater choice for the value of $d$, since there exist many elliptic curves
modulo $n$. We are not restricted to divisors of $n-1$. We can use
any  $d$ that divides the order of any elliptic curve modulo $n$. In 
particular, we  avoid the complication
and the cost coming from the exponent $f$ in $n^f-1$. 
The algorithm remains almost quartic both in time and space.
However, we heuristically  save
 a factor $(\log \log n)^{O(\log\log\log\log n)}$ in the 
complexity.  From a practical viewpoint, it might be worth
choosing for  $d$ a product of prime integers of the appropriate
size, depending of ones  implementation  of fast Fourier 
transform.

Section~\ref{subsection:cyclic} gathers  prerequisites from commutative algebra.
In Section~\ref{subsection:ringextensions}, we describe  a rather
general
variant
of the AKS  primality criterion: it makes uses  of a free
$R$-algebra $S$ of rank $d$ together with 
an $R$-automorphism  $\sigma : S\rightarrow S$ of order $d$.
We recall how this algebra
 can be constructed from multiplicative Kummer
theory as in~\cite{Berri}. 
In Section~\ref{subsection:critere}, we state and prove a primality criterion
involving   rings of elliptic periods.
 The construction of such rings  is detailed in Section~\ref{subsection:construction}.

\subsection{\'Etale cyclic extensions of a field}\label{subsection:cyclic}

Let  $\bK$ be a field and let $\bL\supset \bK$ be a commutative algebra over
$\bK$. We assume $\bL$ is of  finite dimension $d\geqslant 1$ over  $\bK$. We also assume
there exists a  $\bK$-automorphism  $\sigma$ of $\bL$ and 
a $\bK$-basis $(\omega_i)_{i\in \ZZ/d\ZZ}$ of 
$\bL$ such that  $\sigma(\omega_i)=\omega_{i+1}$. So
$\bL$ is a rank $1$ free   $\bK[\cG]$-module,  where $\cG=<\sigma>$ is the
cyclic group generated by  $\sigma$. And
$\omega_0$ is a basis of the  $\bK[\cG]$-module  $\bL$.
In this section,  we recall a few elementary facts about the arithmetic of $\bL$.

First, $\bL$ is a noetherian ring, because it is of finite type over the field
 $\bK$. Further  $\bK$ is the subring $\bL^\cG$ of elements in $\bL$ that are invariant
by $\sigma$. We deduce  \cite[Chapitre 5, paragraphe 1, num{\'e}ro 9, proposition 22]{BourbAC57} 
that   $\bL$ is integral over  $\bK$. Let  $\pgot$ be a prime ideal in $\bL$.
The intersection  $\pgot\cap \bK$ is a prime ideal in  $\bK$, so it is equal to
 $0$. Since  $0$ is maximal in $\bK$, the ideal  $\pgot$ is 
maximal in  $\bL$ 
\cite[Chapitre 5, paragraphe 2, num{\'e}ro 1, Proposition 1]{BourbAC57}.
Thus  $\bL$ is a ring of  dimension $0$. Since $\bL$ is noetherian,
 it is an artinian ring \cite[Chapitre 4, paragraphe 2, num{\'e}ro 5, Proposition 9]{BourbAC57}. Its nilradical $\Ngot$, which 
is equal to its Jacobson radical, is nilpotent. The automorphism
$\sigma$ acts  transitively on the set of prime ideals in $\bL$
\cite[Chapitre 5, paragraphe 2, num{\'e}ro 2, Th{\'e}or{\`e}me 2]{BourbAC57}. 
We denote by  $\cG^Z$ (resp. $\cG^T$) the decomposition  group  (resp. inertia group)
of all these prime ideals. Let  $e\geqslant 1$ be the order of the inertia group  $\cG^T$, and let  
 $f$ be the order of the  quotient $\cG^Z/\cG^T$.
We check that $d=efm$ where  $m$ is the number of   prime ideals in  
$\bL$. Let  $\pgot_0$, $\pgot_1$, \ldots, $\pgot_{m-1}$
be all these prime ideals. They are pairwise relatively prime. The radical of  $\bL$ is
\begin{displaymath}
\Ngot = \bigcap_{0\leqslant i\leqslant m-1}\pgot_i =\prod_{0\leqslant i\leqslant m-1}\pgot_i.
\end{displaymath}

The canonical map 
\begin{displaymath}
\phi : \bL\rightarrow \prod_{0\leqslant i\leqslant m-1}\bL/\pgot_i
\end{displaymath}
is a ring  epimorphism and its kernel is the radical $\Ngot$. 
For every $i$ in $\{0,1, \ldots, m-1\}$, the  quotient $\cG^Z/\cG^T$ is  isomorphic to the group of  $\bK$-automorphisms
of   the residue field $\bM_i=\bL/\pgot_i$  \cite[Chapitre 5, paragraphe 2, num{\'e}ro 2, Th{\'e}or{\`e}me 2]{BourbAC57}.  The field extensions $\bM_i$
of $\bK$
are normal and 
their separable degree   is $f$. 
Let $r$ be their inseparable degree. 
The dimension of the $\bK$-vector space $\bM_i$ is    $rf$. We deduce that
the dimension of  $ \prod_{0\leqslant i\leqslant m-1}\bL/\pgot_i$ is 
$rfm$. And the dimension of the radical  $\Ngot$ is 
\begin{equation}\label{eq:dimrad}
\dim_\bK(\Ngot)=d-rfm=(e-r)fm.
\end{equation}

The radical  $\Ngot$ is  nilpotent: there  exists an integer $k$ such that
 $\Ngot^k=0$. The artinian ring $\bL$ 
is isomorphic  \cite[Theorem 8.7]{AM}  to the product of local artinian rings
$\prod_{0\leqslant i\leqslant m-1} \bL/\pgot_i^k$.

One says that the algebra  $\bL$ is unramified over $\bK$ 
\cite[Chapter 4, Definition 3.17]{Liu}
if the residue fields  $\bL/\pgot_i$ are separable extensions of $\bK$ 
(that is   $r=1$)
 and the local factors  $\bL/\pgot_i^k$ are fields  (e.g. the nilradical
is zero or equivalently  $e-r=0$). This is equivalent to $\bL$  being \'etale
over $\bK$, e.g. the trace form being non-degenerate.

A sufficient condition  for $\bL$ to be unramified over $\bK$
is that 
for every prime divisor   $\ell$ of $d$ there exists an element $a_\ell$ in $\bL$ 
such that  $\sigma^{D/\ell}(a_\ell)-a_\ell$ is a unit. Indeed this proves that 
$\sigma^{D/\ell}$ does not lie in $\cG^T$.  So $e=1$. And
$r=1$ also,  using Eq.~\eqref{eq:dimrad}.

Assume now $\bK$ is a finite field and $\bL$ is reduced (therefore
\'etale over $\bK$).
Remember  $\pgot_0$, $\pgot_1$, \ldots, $\pgot_{m-1}$ are  the  prime ideals in  $\bL$. The Frobenius automorphism
$\Phi_i$ of $\bM_i=\bL/\pgot_i$ is the reduction modulo $\pgot_i $ of some power $\sigma^{z_i}$ of $\sigma$ lying in $\cG^Z$. Especially, for every $a$
in $\bL$, one has $\sigma^{z_0} (a)=a^p\bmod \pgot_0$ for some integer $z_0$. We let $\sigma$
act on the above congruence and deduce that $z_0=z_1=\cdots=z_{d-1}$ because $\sigma$ acts transitively on
the set of primes. So there exists an integer $z$ such that for every element $a$ in $\bL$ we have
\begin{equation}
  \label{eq:expp}
  a^p =\sigma^z(a)\,.
\end{equation}
Of course, $z$ is a multiple of $m$.

\subsection{Ring extensions and primality proving}\label{subsection:ringextensions}

Let $n\geqslant 2$ be an integer and set
$R=\ZZ/n\ZZ$. In this section, we state a general AKS-like primality criterion in terms
of the existence of some commutative free $R$-algebra $S$ of finite rank
fulfilling simple conditions.

Let $S\supset
R$ be a finite
free commutative $R$-algebra of rank $d\geqslant 1$. Then $R$ can be
identified with  a subring of $S$.   
Let $\sigma : S \rightarrow S$ be an
$R$-automorphism of $S$ and assume that there exists an $R$ basis $(\omega_i)_{i\in \ZZ/d\ZZ}$
of $S$ such that $\sigma (\omega_i)=\omega_{i+1}$. 
Let $p$ be a positive prime integer dividing $n$.  Set
$\bL=S/pS$ and $\bK=R/pR=\ZZ/p\ZZ$. Assume  $\bL$ is reduced.
This is always the case when $S$ is \'etale  over $R$
\cite[Chapter 4, Definition 3.17, Lemma 3.20]{Liu}.
 The $R$-automorphism 
$\sigma : S\rightarrow S$ induces a $\bK$-automorphism of $\bL$ that we call
$\sigma$ also.
Let $\theta$ be a unit in $S$ such that
\begin{equation*}
  \theta^n =\sigma(\theta)\,.
\end{equation*}
Reducing this identity modulo $p$ and setting $a=\theta\bmod p\in \bL$,  we obtain 
\begin{equation}
  \label{eq:expnmodp}
  a^n =\sigma(a)\,.
\end{equation}
Using Eqs.~(\ref{eq:expnmodp}) and~(\ref{eq:expp}) repeatedly, we prove that there exists an integer $z$ such that 
 for $k,\, l \in \NN$, we have
  \begin{equation}\label{eq:comb}
    a^{n^kp^l}=\sigma^{k+zl}(a).
  \end{equation}

Let $\pgot$ be a prime ideal in $\bL$ and set $\bM=\bL/\pgot$.
Set  $b= a\bmod \pgot\in \bM$.
Let $G\subset \bL^*$ be the group generated by $a$ and let $H\subset \bM$
be the group generated by $b$. 
We first show that the reduction modulo $\pgot$ map  $G\rightarrow H$ is a bijection.
Indeed, let $k$ be a positive integer such that $b^k=1 \in \bM$. Then $a^k=1\bmod \pgot$. 
We raise both members in this congruence to the $n$-th power.
 Using Eq.~(\ref{eq:expnmodp}),  we find 
$a^{kn}=a^{nk}=\sigma(a)^k=\sigma(a^k)=1\bmod \pgot$. So $a^k=1\bmod \sigma^{-1}(\pgot)$.
We remind that $\sigma$ acts transitively on the set of primes in $\bL$. So
$a^k$ is congruent to $1$ modulo all these primes. Since $\bL$ is reduced, we deduce
that  $a^k=1$.

The group $H$ is a subgroup of $\bM^*$. Therefore the order $h$
 of $H$ (which is the order
of $G$ also) divides $p^f-1$ where $f$ is the dimension of $\bM$ over $\bK$.
It is  thus  clear that $p$ and $\# H$ are coprime. Iterating $d$ times Eq.~(\ref{eq:expnmodp}), we find
that $a^{n^d}=a$. So $n$ also is invertible modulo $h=\#G=\#H$. So Eq.~(\ref{eq:comb}) makes
sense and holds true for $k$ and $l$ in $\ZZ$, provided the exponents are seen
as residues modulo $h$.

  We set $q=n/p$ and from Eqs.~\eqref{eq:expnmodp} and~\eqref{eq:expp}, we
  deduce that $a^q = \sigma^{1-z}(a)$ \,. Moreover, there exist four
  integers $i$, $i'$, $j$ and $j'$ in $\{0,1, \ldots, \lfloor \sqrt d \rfloor
  \}$ such that  $(i,j)\ne(i',j')$ and $i(1-z)+jz$ is congruent to
  $i'(1-z)+j'z$ modulo $d$.  Setting in Eq.~(\ref{eq:comb}), first $k=i$ and
  $l=j-i$, and then $k=i'$ and $l=j'-i'$, we find that exponentiations by
  $q^ip^{j}$ and $q^{i'}p^{j'}$ act similarly on $a$. We deduce that
  \begin{equation}\label{eq:congrug}
    q^ip^j =   q^{i'}p^{j'}\bmod \# G.
  \end{equation}
  We now observe that 
 both integers $q^ip^j$ and $q^{i'}p^{j'}$ are bounded
  above by $n^{\lfloor \sqrt d \rfloor }$. 
If 
\begin{equation*}
n^{\lfloor \sqrt d \rfloor }\leqslant \# G\,,
\end{equation*}
then 
  Congruence~(\ref{eq:congrug}) is an equality between integers and we deduce
  that $n$ is a power of $p$.

\begin{theorem}[AKS criterion]\label{th:AKSgen}
Let $n\geqslant 2$ be an integer and set $R=\ZZ/n\ZZ$.
 Let $S\supset R$ be a 
free 
algebra of rank 
 $d$ over $R$.  Let  $\sigma$ be an $R$-automorphism of
$S$. Let $\cG$ be the group generated by $\sigma$.
Assume  $S$ is a  free $R[\cG]$-module of rank 
$1$: there exists an element $\omega$ in $S$ such that 
$(\omega, \sigma(\omega), \ldots, \sigma^{d-1}(\omega))$
is an $R$ basis of $L$. Let $\theta$ be a unit in $S$ such that
$\theta^n=\sigma(\theta)$. Let 
$p$ be a prime divisor  of $n$.
Assume   $S/pS$ is reduced and 
$\theta \bmod p$
generates a subgroup of order 
 at least $n^{\lfloor \sqrt d \rfloor }$ in $(S/pS)^*$. 
Then $n$ is a  power of $p$.
\end{theorem}

The condition that $S/pS$ is reduced is granted  if
$S$ is \'etale  over $R$.
A sufficient condition for $S$
to be \'etale  over $R$
 is that for every prime 
divisor  $\ell$ of $d$,  there exists an element $a_\ell$ in $S$ 
such that  $\sigma^{D/\ell}(a_\ell)-a_\ell$ is a unit. 

The condition on the size of the group generated by $\theta \bmod p$ is often
obtained with the help of geometric arguments. In our cases, these are degree
considerations, which yield a lower bound for $d$.  \bigskip

Berritzbeitia,   Cheng, Avanzi,   Mih\u{a}ilescu and
 Bernstein construct $S$ as $R[x]/(x^d-\alpha)$ where $d\geqslant 2$ divides $n-1$ and 
$\alpha$ is a unit in
$R$. 
 We set $n-1=dm$ 
and
$\zeta=\alpha^{m}$. Assume $\zeta$ has exact order $d$ in $R^*$. This means that
$\zeta^d=1$ and
$\zeta^k-1$ is a unit for every $1\leqslant k < d$. We define an $R$
automorphism $\sigma : S\rightarrow S$ by setting $\sigma(x)=\zeta x$.
We set $\omega=(\alpha-1)/(x-1)=1+x+x^2+\cdots+x^{d-1}\bmod x^d-\alpha$ and we check
that $(\omega, \sigma(\omega), \ldots, \sigma^{d-1}(\omega))$
is an $R$-basis of $S$. Indeed $(1,x,x^2, \ldots, x^{d-1})$ is a basis, and
the matrix connecting the two systems is a Vandermonde matrix
$V(1,\zeta, \ldots, \zeta^{d-1})$ which is invertible since $\zeta$ has exact order $d$. 
So $S$ is a free $R[\sigma]$-module of rank $1$. 

We note that $x\bmod x^d-\alpha$ is a unit in $S$
because $\alpha$ is a unit in $R$.
For every integer $1\leqslant k < d$, the difference 
$\sigma^k(x)-x=(\zeta^k-1)x$ is a unit in $S$, because $\zeta$ has exact order $d$.
So $S$ is \'etale  over $R$.
The main computational step in
Berrizbeitia test is to check, by explicit calculation, that the following
congruence holds true in $S$,
\begin{equation}\label{eq:berricong}
  (x-1)^n=\zeta x-1 \bmod (n,x^d-\alpha)\,. 
\end{equation}

So, we set $\theta=x-1\bmod (n,x^d-a)$. This is a unit in $S$ because $\alpha-1$ is a unit in $R$.
Letting $\sigma$ repeatedly act on Eq.~(\ref{eq:berricong}), we deduce that for any
positive integer $k$, the class $\zeta^k x-1\bmod (n,x^d-\alpha)$ is a power of $\theta$.

Let $p$ be any prime divisor of $n$.
We set $a=\theta\bmod p=x-1\bmod (p,x^d-\alpha) \in S/pS$.
We  show  that
the order of $a$
in $(S/pS)^*$ is large.
For every subset $\cS$ of $\{0,1,\ldots, d-1\}$, we denote by $a_\cS$ 
the product
\begin{displaymath}
\prod_{k\in \cS}(\zeta^kx-1) \bmod (p,x^d-a)=\prod_{k\in \cS}\sigma^k(a).
\end{displaymath}
This is a power of $a$, because every $\sigma^k(a)$ is.
Degree
considerations similar to those in the original paper~\cite{AKS} show that 
if $\cS_1$ and $\cS_2$ are two
strict distinct subsets of $\{0,1,\ldots, d-1\}$,   then $a_{\cS_1}$ and $a_{\cS_2}$ are distinct elements
in  $S/pS$.
So the
order of $a$ in $(S/p S)^*$ is at least $2^d-1$. This lower bound can be improved by
several means (see for instance Voloch's work~\cite{voloch}).

If  $2^d$ is bigger than $n^{\lfloor \sqrt d \rfloor }$, we deduce from
Theorem~\ref{th:AKSgen} that $n$ is a  prime power.

\begin{corollary}[Berrizbeitia criterion]\label{lemma:AKSber}
Let $n\geqslant 3$ be an integer and set $R=\ZZ/n\ZZ$.
Let $S=R[x]/(x^d-\alpha)$ where $d\geqslant 2$ divides $n-1$. 
 Set $n-1=dm$ and
 assume $\zeta=\alpha^{m}$  has exact order $d$ in $R^*$. 
Assume Eq.~(\ref{eq:berricong}) holds true in $S$.
If $2^d$ is bigger than $n^{\lfloor \sqrt d \rfloor }$, then $n$ is
a prime power.
\end{corollary}

\medskip

In Section~\ref{subsection:critere}, we adapt this construction to the broader
general context of Kummer theory of elliptic curves. This way, we get rid of
the condition that $d$ divides $n-1$.

\subsection{A primality criterion}\label{subsection:critere}

In this section, we state and prove a primality criterion
involving elliptic periods.
Assume we are given an  integer $n\geqslant 2$. We
set $R=\ZZ/n\ZZ$ and we assume we 
are in the situation of Theorem~\ref{th:tensor}.
We are given  a Weierstrass elliptic curve  $E$ 
over $R$, a  positive  integer 
$d$ relatively prime to $2n$ and  a section 
$T\in E(R)$ of exact order $d$. 
The quotient by $\langle T\rangle$ 
isogeny $I : E\rightarrow E'$ is given by V{\'e}lu's formulae.
We are given a section $A\in E'_{\text{aff}}(R)$ and we call
\begin{equation*}
\Fgot_A=(x'-x'(A),y'-y'(A))
\end{equation*}
the ideal of $I^{-1}(A)$
in $R[x,y,{1}/{\psi_d(x)}]/(\Lambda(a_1,a_2,a_3,a_4,a_6,x,y))$.
We assume that  $D(x'(A))$ is a  unit in $R$, where $D$ is defined 
in Eqs.~(\ref{eq:Dinit}), (\ref{eq:Dpsi}) and (\ref{eq:psiI}).
Let  
\begin{equation*}
S=R[x,y,{1}/{\psi_d(x)}]/(x'-x'(A),y'-y'(A))
\end{equation*}
be the residue ring of 
 $R[x,y,{1}/{\psi_d(x)}]/(\Lambda(a_1,a_2,a_3,a_4,a_6,x,y))$ at $I^{-1}(A)$.

We call  $\sigma : S\rightarrow S$  the automorphism induced
on $S$   by  the translation $\tau_{-T}$: 
\begin{displaymath}
  \begin{array}{crcl}
    \sigma :& S & \longrightarrow  &   S\,,\\
    &f\bmod \Fgot_A& \longmapsto  &f\circ\tau_{-T}\bmod \Fgot_A\,.
  \end{array}
\end{displaymath}

For $k\in \ZZ/d\ZZ$, we set $\theta_k = u_k \bmod \Fgot_A$.
The $(\theta_k)_{k\in \ZZ/d\ZZ}$ form an $R$-basis of $S$  and 
we have $\sigma(\theta_k)=\theta_{k+1}$. The algebra $S$ is finite
free \'etale of rank $d$ over
$R$ because the determinant $D(x'(A))$ of the trace form is a unit.
The main computational step now
is to check, by explicit calculation, that the following
congruence holds true in $S$,
\begin{equation}\label{eq:ellcong}
  \theta_0^n=\theta_1\,. 
\end{equation}

Letting $\sigma$ repeatedly act on Eq.~(\ref{eq:ellcong}), 
we deduce that for any
$k\in \ZZ/d\ZZ$, $\theta_k$ is a power of $\theta_0$.
In particular, all  $\theta_k$ belong to  the ideal generated by 
$\theta_0$.  Using Eq.~(\ref{eq:sum1}), we deduce that
$1=\sum_{k\in \ZZ/d\ZZ}\theta_k$ belongs to 
the ideal generated by $\theta_0$. So $\theta_0$ is a unit.

Let $p$ be any prime divisor of $n$.
We set $a=\theta_0 \bmod p \in S/pS$.
We  show  that
the order of $a$
in $(S/pS)^*$ is large.
To every
 subset $\cS$ of $\ZZ/d\ZZ$, we  associate the product 
\begin{equation*}
u_\cS=\prod_{k\in \cS} u_k
\end{equation*}

We note  that $u_\cS\bmod (\Fgot_A,p)=\prod_{k\in \cS} (\theta_k\bmod p)$
is a power of $a$. 
Let  $\cS_1$ 
and $\cS_2$ be  two  subsets  of  
\begin{displaymath}
\{0,2,4, \ldots, d-3\}\subset
\ZZ/d\ZZ.
\end{displaymath}
Let $l_1$ and  $l_2$ be two integers that
are relatively prime to $p$.
Then $l_1u_{\cS_1}\not =l_2u_{\cS_2}\bmod (\Fgot_A,p)$ unless $\cS_1=\cS_2$
and $l_1=l_2\bmod p$. 
Indeed, if $l_1u_{\cS_1} =l_2u_{\cS_2}\bmod (\Fgot_A,p)$ then $l_1u_{\cS_1}-l_2u_{\cS_2}
\bmod p$ is a function on $E\bmod p$  with divisor 
$\geqslant -\sum_{k\in \ZZ/d\ZZ}[kt]$  
and it cancels
on the degree $d$ divisor $I^{-1}(A)\bmod p$.  So 
$l_1u_{\cS_1} =l_2u_{\cS_2}\bmod p$. Therefore
 these two functions have the same poles.
We  deduce first, that $\cS_1=\cS_2$, and then, that $l_1=l_2$.

There  are $2^{\frac{d-1}{2}}$ subsets of $\{0,2,4, \ldots, d-3\}$. 
So, the order of $a$ in $(S/pS)^*$ is at least  $2^{\frac{d-1}{2}}$.
\medskip

Using Theorem~\ref{th:AKSgen}, we deduce the following primality criterion.
\begin{corollary}[Elliptic AKS criterion]\label{th:ellaks}
Let $n\geqslant 2$  be an integer and let $E$
be a Weierstrass elliptic curve over $R=\ZZ/n\ZZ$.
Let $T\in E(R)$ be a section of exact order $d$ where
$d$ is an integer relatively prime to $2n$.
 Let $E'$
be the quotient $E/\langle T \rangle$ given by V{\'e}lu's formulae.
Let $A\in E'_{\text{aff}}(R)$ be a section such that the vector
$\ve = \left( \egot_{k}(A) \right)_{k}$ defined
by Eq.~(\ref{eq:egot}) is 
invertible
for the convolution product $\star$ on $R^d$. 

Assume that
\begin{equation}\label{eq:cong}
(\theta_0)^n=\theta_1
\end{equation}
holds true in the ring of elliptic periods $S=R[x,y,{1}/{\psi_d(x)}]/(x'-x'(A),
y'-y'(A))$.

Assume further that 
\begin{equation}\label{eq:boundd}
2^{\frac{d-1}{2}}\geqslant n^{\sqrt{d}}.
\end{equation}
Then $n$ is a prime power.
\end{corollary}

We recall  that the condition that the vector 
$\ve$ be invertible means that the section $A$ does not cross
the kernel of the dual isogeny $I' :E'\rightarrow E$.
Checking Eq.~(\ref{eq:cong}) requires $O(\log n)$ multiplications 
in the ring $S$. Any such multiplication
requires $O(d\log d\log\log d)$ operations (additions,
subtractions, multiplications)  in $R=\ZZ/n\ZZ$. 
So the total cost is 
\begin{displaymath}
O((\log n)^2(\log \log n)^{1+o(1)} \times d\log d\log\log d)
\end{displaymath}
elementary operations using fast arithmetic~\cite{Schon,SchonStras}.
In Section~\ref{subsection:construction}, we explain why one can hope to find
a degree $d$ that is $O((\log n)^2)$. With such a $d$, one can
verify Eq.~(\ref{eq:cong}) in time
\begin{displaymath}
O((\log n)^4(\log \log n)^{2+o(1)}).
\end{displaymath}
Moreover, we explain how
to construct the ring $S$ in Corollary~\ref{th:ellaks}.

\subsection{Construction of a ring of elliptic periods}\label{subsection:construction}

In this section, we explain how to  
construct the  ring of elliptic periods that is required
 to prove that a given integer
$n\geqslant 2$ is prime using   Corollary~\ref{th:ellaks}.
So, we are given  an  integer
 $n\geqslant 2$ which is probably prime: it already passed many
pseudo-primality tests.
We want to construct a ring of elliptic
periods modulo $n$ with  rank  $d$ for some $d$ satisfying
Inequality~(\ref{eq:boundd}). A sufficient condition is that $d\geqslant \dmin$ with
\begin{displaymath}
  \dmin =\lceil  4(\log_2 n)^2+2\rceil\,.
\end{displaymath}

We assume that $d$ is odd too.
We like $d$ to be as small as possible.
We set $\dmax=\dmin \times O(1)$ and ask that $d\in [\dmin, \dmax]$.
The  construction  is probabilistic  and 
relies on several heuristics.
Since $n$ is probably prime, we shall allow ourselves 
to use  algorithms that are only proven to work under the condition
that $n$ is prime.  This is not an issue  as far as we can  check
the result rigorously (and efficiently).

We set $R=\ZZ/n\ZZ$.
We want to construct an elliptic curve $E$
 over $R$ with
a section $T\in E(R)$ of exact order $d$ in the sense of
\cite[Chapter 1, 1.4]{K}. 
We use complex multiplication theory.
\bigskip

\noindent
\textit{The first step of the algorithm  selects quadratic imaginary orders.}
We look over the maximal quadratic imaginary orders $\cO$ for decreasing
fundamental discriminants $-\Delta$. We start with $-\Delta=-7$.
For each  order $\cO$, we first look for a square root $\delta$
of $-\Delta$ modulo
$n$ using the algorithm of Legendre. Since $n$
is expected to be prime, the algorithm will succeed in probabilistic 
time $(\log n)^2(\log\log n)^{1+o(1)}$. And of course we can check
the result rigorously in time
$(\log n)$ $(\log\log n)^{1+o(1)}$.
For a given $n$, such a square root
exists for one quadratic order over two. If we fail to find such a
square root, we go to the next quadratic order. 
\smallskip

Once we have found  a square root $\delta$ of $-\Delta$ modulo $n$, we call
$\ngot$ the ideal $( n,\sqrt{-\Delta}-\delta )$ in $\cO$ and we 
look for
an element  with norm $n$ in $\ngot$. We use fast Cornachia's
algorithm. It runs in deterministic time $(\log n)(\log \log n)^{2+o(1)}$ 
and finds such an element  $\phi\in \cO$ when  it exists. 
\smallskip

We then set $\tgot =\Tr(\phi)$ and look for an integer $d$ that satisfies the
following conditions:
\begin{itemize}[itemsep=0.5pt,parsep=0.5pt,partopsep=0.5pt,topsep=0.5pt]
\item $d \in [\dmin, \dmax]$,
\item $d$ is relatively prime to  $n(n-1)(n+1)$,
\item there exists an $\epsilon \in \{1,-1\}$ such that $d$ divides
  $n+1-\epsilon\tgot$ and is relatively prime to $(n+1-\epsilon\tgot )/d$.
\end{itemize}
\smallskip

In order to  find such a $d$, we apply the elliptic curve factoring
method to $n+1-\tgot$ and $n+1+\tgot$. Since the factors we are looking
for are very small, we expect to find them in time $(\log n)^{1+o(1)}$.
If we find  no such integer $d$, we go to the next fundamental discriminant 
$-\Delta$.

We expect to succeed in finding
an integer $d$ for some $\Delta=(\log \log n)^{2+o(1)}$.
Also the expected running time of this first step
is $(\log n)^{2+o(1)}$.
%
%
We note that the search for split discriminants can be accelerated using the
same technique as in the J.O. Shallit  fast-ECPP algorithm
\cite{hand,morainfastecpp}.\medskip
\bigskip

\noindent
\textit{The second step of the algorithm  constructs the ring $S$ from the  couple $(-\Delta,d)$.}
Once we have found a  quadratic order   $\cO$,  we compute
the associated  Hilbert class polynomial. Computing 
$H_\cO(X)$ requires  quasi-linear time in the size 
of this polynomial. This polynomial has degree 
$\Delta^{1/2+o(1)}$ and height $\Delta^{1/2+o(1)}$,  where $-\Delta$
is the discriminant of $\cO$. So $H_\cO(X)$ 
can be computed  in time $\Delta^{1+o(1)}$.
Finding a root $j$ of $H_\cO(X)$ modulo $n$  is achieved in
probabilistic time 
\begin{displaymath}
\Delta^{1/2+o(1)} (\log n)^{2+o(1)}\,.
\end{displaymath}
So the  time for finding this root will be
$(\log n)^{2+o(1)}$.
\smallskip

Once computed a root of the modular
polynomial, we construct an elliptic curve $E$ over $R=\ZZ/n\ZZ$
having  modular invariant $j$. 
We then construct a random  $R$-section $P$
on $E$. 
We expect one and only one among $[n+1-\tgot]P$ and $[n+1+\tgot]P$ 
to be  equal to the zero section $O$. If this is not the case, we pick
another point $P$.
Let $\epsilon \in \{-1,1\}$ be such 
that $d$ divides $n+1-\epsilon \tgot$.
If we have found a  section  $P$ such that 
$[n+1-\epsilon \tgot]P\not =O$, then we replace
$E$ by its quadratic twist. And we start again
with this new curve.
If we have found a point $P$ such that 
$[n+1-\epsilon \tgot]P=O$ and $[n+1+\epsilon \tgot]P\not =O$, then we multiply 
$P$ by $(n+1-\tgot)/d$ and obtain a section $T$ that, we hope,
has exact order $d$. 
We can  test that $T$ has exact order $d$  by checking 
that $\psi_k(x(T))$ is a unit in $R$ for every strict divisor $k$ of
$d$.
If this condition does not hold, we pick another section $P$ on $E$.
\smallskip

Once we have found a $T$ of exact order $d$, we consider  the
quotient isogeny $I : E\rightarrow E'$. We
compute the coefficients in the Weierstrass equation of $E'$
thanks to Eq.~(\ref{eq:coeff'}). We do not write down explicit 
equations for $I$. 
We look for a $R$-section  $A$ on $E'$ having exact order $d$.
We let $S$ be the residue
ring of  $I^{-1}(A)$. Elements in $S$ are represented
by vectors in $R^d$. The automorphism $\sigma$ is the  cyclic shift 
of coordinates. There remains to describe the multiplication 
law. To this end, we pick  an auxiliary 
$R$-section $M$  of $E$ such that $N=I(M)$ does not cross 
the kernel of the dual isogeny $I'$; or equivalently  $D(x'(N))$
is a unit in $R$.
We now can  compute the multiplication tensor of the  ring $S$.
This tensor is given by Theorem~\ref{th:tensor}. We just need 
to compute the vectors 
$\hviota$, $\vuN$, $\vxN$ using the method given
in Section~\ref{subsubsection:dual}.  
This requires $O(d(\log d)^2\log \log d)$ operations in $R$.
This finishes the construction of the ring $S$.
\smallskip

The expected running time of this second step is 
$  (\log n)^{1+o(1)}(\log n+d^{1+o(1)})=(\log n)^{3+o(1)}$
operations in $R$.
\ifsubmitted\else
\bigskip

\noindent
\textit{Remark.} To improve memory requirements of the algorithm, we may try
to replace the degree $O(\log^2 n)$ extension $S$ by a direct product of
$O(\log n)$ extensions $S_k$, each of degree $d_k=O(\log n)$ and each endowed
with an $R$-automorphism $\sigma_k$ of order $d_k$. Unfortunately, this
product is endowed with an $R$-automorphism, $\prod_k \sigma_k$, of order
$\prod_k d_k$, much larger than the rank $\sum_k d_k$ and this is a serious
drawback to get an efficient primality criterion.
\fi

\subsection{Example}

We consider here a primality test for $n=1009$.\medskip

We first notice that $\dmin = \lceil 4(\log_2 n)^2+2\rceil = 401$, and a quick
search among maximal quadratic imaginary orders $\cO$ for decreasing
fundamental discriminants yields $d=479$ for $-\Delta = -148$ (and class
number 2).  In truth, we have
\begin{math}
  52^2+3^2\,148 = 4\,n\,,
\end{math}
and the corresponding elliptic curve has got $n+1-52$ ($= 2\times
479$) points.

The Hilbert class polynomial associated to $-\Delta = -148$ is
\begin{displaymath}
  H_{-148}(X) = X^2 - 39660183801072000\,X - 7898242515936467904000000\,.
\end{displaymath}
One of its roots mod $n$ is $j_E = 353$, and one can check that the point
$T=(296, 432)$ is of order $d$ on the elliptic curve
\begin{displaymath}
  E~: {y}^{2}+xy={x}^{3}+364\,x+907.
\end{displaymath}
Similarly, we can check that the point $M=(726, 695)$ is of order 958. V\'elu's formulae yield
 then the quotient elliptic curve,
\begin{displaymath}
  E / \langle T\rangle~: {y}^{2}+xy={x}^{3}+130\,x+233\,.
\end{displaymath}
We choose $A = (383, 201)$, a point of order $d$ on $E / \langle T\rangle$. We can check
also that the image of $M$ by the isogeny is equal to $N=(321, 344)$, a point
of order 2.\medskip

With this setting, we can now define, without any ambiguity, a normal elliptic
basis $\Theta= (\theta_{k})_{k \in \ZZ/d\ZZ}$ (see
Section~\ref{subsection:ellp}) and a final computation yields
\begin{displaymath}
  \theta_0^{1009} = \theta_{91}\,. 
\end{displaymath}

We check that $91$ is relatively prime to $479$. So $T'=91T$ is a point
of exact order $479$. Applying Corollary~\ref{th:ellaks} with
$T'$ instead of $T$, we prove that  $1009$ is a prime.


\ifsubmitted\else

\section{A stronger criterion}\label{section:strong}

We now improve on the primality criterion of Section~\ref{section:AKSTEST},
at the expense of some more  geometry and combinatorics. 
If we come back to the proof of Corollary~\ref{th:ellaks}, we find
ourselves with an elliptic curve $E$ over a field $\bK=\Fp$. We are
given a point $T$ of odd order $d\geqslant 3$ and
the corresponding  automorphism $\sigma$ of the field
of functions,
\begin{displaymath}
  \begin{array}{crcl}
    \sigma :& \bK(E) & \longrightarrow  &   \bK(E)\,,\\
    &f & \longmapsto  &f\circ\tau_{-T}\,.
  \end{array}
\end{displaymath}

We also are given a function $u_0$ on $E$. We have an isogeny $I :
E\rightarrow E'$, a 
divisor $\Ker I = [O]+[T]+[2T]+\dots+[(d-1)T] $  and
the associated  $\bK$-linear space $\cL(\Ker I)$
of dimension $d$ lying  inside $\bK(E)$. We consider the $\ZZ[\sigma]$-module $\cU$ generated by $u_0$ inside $\bK(E)^*$. The essential
point is that the intersection $\cU\cap \cL(\Ker I)$ is  large: the quotient
$(\cU\cap \cL(\Ker I))/\bK^*$ has cardinality at least $2^{\frac{d-1}{2}}$. All
the functions in this intersection have degree $\leqslant d$.
We want to  replace $u_0$ by a slightly different function and 
obtain an even larger set of functions with small degree in  the corresponding
monogenous 
$\ZZ[\sigma]$-module.\medskip

This section is organized as follows. Section~\ref{subsection:units}
studies  the structure of the $\ZZ$-module $\cU=\bK[E-\langle T\rangle]^*$
of units in
$\bK[E-\langle T\rangle]$. We show that the quotient module $\cU/\bK^*$
 is monogenous
as a $\ZZ[\sigma]$-module and we exhibit a generator for it.
Section~\ref{subsection:inter} gives a lower bound for the
number of functions with degree $\leqslant (d-1)/2$ in  $\cU/\bK^*$.
The resulting strengthened primality criterion (Corollary~\ref{cor:AKSS})
is stated in  Section~\ref{subsection:strong}. 
It is asymptotically twenty five  times faster than the test
resulting from Corollary~\ref{th:ellaks}. \medskip

We postpone to Appendix~\ref{section:appendix} some of the technical results
needed in the proof of the stronger primality criterion. The determinant
needed in Section~\ref{subsection:units} is calculated in
Section~\ref{subsection:det}. Section~\ref{subsection:binom} gives a simple
lower bound for binomial coefficients that is useful to prove in
Section~\ref{subsection:enum} a combinatorial lemma.

\subsection{A group of elliptic units}\label{subsection:units}

Let $\bK$ be a field  and $E$ an elliptic curve over $\bK$.
Let $d\geqslant 3$ be an odd integer and let $T$ be a point of order
$d$ in $E(\bK)$. 
Let $\sigma : \bK(E) \rightarrow \bK(E)$
be the automorphism that sends $f$ to $f\circ \tau_{-T}$.
In this paragraph, we are interested in the group $\cU$
of functions in $\bK(E)$ having no zeros nor poles outside
the group $\langle T\rangle$ generated by $T$.

There is a unique multiple $\hat T$ of $T$
such that $T=2\HT$.
For every $k$ in $\ZZ/d\ZZ$, we define $u_k$
as in Section~\ref{subsubsection:ENB}. This is a function
having two simple poles: one at $kT$ and one at
$(k+1)T$. 
If $l=2k\bmod d$, we set $\hu_{l}=u_k-u_0(\HT)=u_k-u_k(kT+\HT)=\hu_0\circ \tau_{-l\HT}$.
Its divisor is
\begin{displaymath}
(\hu_{l})=-[kT]+2[\HT+kT]-[(k+1)T]= -[l\HT]+2[(l+1)\HT]-[(l+2)\HT]
\end{displaymath}
and it is clear that 
\begin{equation}\label{eq:produk}
\prod_{k\in \ZZ/d\ZZ}\hu_k\in \bK^*\,.
\end{equation}

 We
want to prove that the $\hu_k$ generate the lattice
$\cU/\bK^*$, or equivalently that $(\hu_k)_{0\leqslant k\leqslant d-2}$ is a 
$\ZZ$-basis
for it.
Let $\cV$ be the submodule  of $\ZZ^d$ consisting of vectors
$(e_k)_k$ such that $\sum_{k\in \ZZ/d\ZZ}e_k=0$.
Let $\cW$ be the sublattice  of $\cV$ consisting of vectors
$(e_k)_k$ such that $\sum_{k\in \ZZ/d\ZZ}e_k=0$ 
and $\sum_{k\in \ZZ/d\ZZ}ke_k =0\bmod d$. 
The index of $\cW$ in $\cV$ is $d$.
We  construct a bijection
\begin{equation}\label{eq:V}
V : \cU/\bK^* \rightarrow \cW
\end{equation}
by associating to every unit $u$
the vector consisting of its valuations at all points
$k\HT$ for $k\in \ZZ/d\ZZ$.
In order to prove that $(\hu_k)_{0\leqslant k\leqslant d-2}$ is a 
$\ZZ$-basis
for $\cU/\bK^*$, we consider  the following $(d-1)\times d$
matrix,
\begin{equation}\label{eq:vandermonde}
\left( \begin{array}{ccccccc}
-1&2&-1&0&\cdots&0&0\\
0&-1&2&-1&\cdots&0&0\\
0&0&-1&\ddots&\ddots&\vdots&\vdots\\
\vdots&\vdots&\vdots&\ddots&2&-1&0\\
0&0&0&\cdots&-1&2&-1\\
-1&0&0&\cdots&0&-1&2
\end{array} \right)
\end{equation}

We stress that 
the $d-1$  lines in this matrix are the images $V(\hu_k)$ 
of the $\hu_k$ by $V$, for $0\leqslant k\leqslant d-2$. We want to show
that these lines
 form a basis
of $\cW$. We call them $W_k$ for $0\leqslant k\leqslant d-2$.
From equation~(\ref{eq:det}) below,  we deduce that the determinant
of the rightmost $(d-1)\times (d-1)$ minor in
the above matrix is $d$. So the index of the lattice generated by
the $(W_k)_{0\leqslant k\leqslant d-2}$ inside $\cV$ is a divisor of $d$. 
This implies that  this lattice is equal to $\cW$.

\begin{lemma}
Let $\cU \subset \bK(E)^*$  be the group of functions having
no zero nor pole outside the subgroup $\langle T\rangle$ generated by $T$.
Then $\cU/\bK^*$ is a free $\ZZ$-module and $(\hu_k)_{0\leqslant k \leqslant d-2 }$
is a basis for it. As a $\ZZ[\sigma]$-module, $\cU/\bK^*$
is monogenous and $\hu_0$ is a generator for it.
\end{lemma}

\subsection{Elliptic units with small degree}\label{subsection:inter}

In this paragraph, we are interested  in the subset $\cT$ of $\cU$
consisting of functions in $\cU$ having degree $\leqslant (d-1)/2$. 
Recall the definitions of $\cV$ and  $\cW$ given in
Section~\ref{subsection:units}.
Let
$\cI$ be the subset of the lattice $\cV$ consisting of vectors
having $L^1$-norm $\leqslant d-1$.
Let $\cJ$ be the intersection of $\cI$ and $\cW$.
The set $\cT /\bK^*$ is mapped  bijectively
onto  $\cJ$ by the map  $V$ defined in Eq.~(\ref{eq:V}).  
We want to bound from below the cardinality of $\cJ$.

For every  $k$ and $l$ in $\ZZ/d\ZZ$, 
the map $\kappa_{k,l} : \cV\rightarrow  \cV$  is defined to
be the map
that increments the $k$-th coordinate and decrements the 
$l$-th one.
There are $d(d-1)+1$ such maps. We fix an arbitrary  total
order on the set consisting of these $d(d-1)+1$ maps.
For every vector $\vvv=(v_i)_{i\in \ZZ/d\ZZ}$ in $\cI$, there is at least one
map $\kappa_{k,l}$ such that $\kappa_{k,l}(\vvv)$ is in
$\cJ$: 
\begin{itemize}
\item if $\vvv$ is already in $\cW$, we apply the identity
$\kappa_{0,0}$ to $\vvv$; 
\item 
otherwise, 
we assume for instance that the $l$-th coordinate is positive.
We set $k=l-\sum_{i\in \ZZ/d\ZZ}iv_i \bmod d$ and we can check that
$\kappa_{k,l}(\vvv)$ is in $\cW$ and its norm is not bigger than
the norm of $\vvv$.
\end{itemize}

For every vector $\vvv$ in $\cI$, we call $\kappa (\vvv)$
the image of $\vvv$ by the smallest map 
$\kappa_{k,l}$ such that $\kappa_{k,l}(\vvv)$ is in
$\cJ$. This way, we define a map $\kappa : \cI\rightarrow \cJ$.
Every element in $\cJ$ has at most $d(d-1)+1$ preimages by $\kappa$.
Therefore, the sizes of $\cI$ and $\cJ$ are related by the following 
inequation,
\begin{equation*}\label{eq:IJ}
\# \cJ\geqslant \frac{\# \cI}{d^2}\,.
\end{equation*} 

We know from Lemma~\ref{lemma:count} that $\log \# \cI \geqslant 1.74498\times d$ if
$d\geqslant 2001$. We deduce that $\log \#\cJ\geqslant (1.74498 -0.0076)\times
d$ in this case. Hence, we have the following
lemma.

\begin{lemma}\label{lemma:bound}
If $d\geqslant 2001$ is an odd integer, the set $\cT/\bK^*$
consisting of elliptic units (modulo constants) having degree 
$\leqslant (d-1)/2$ has cardinality
\begin{equation*}\label{eq:bound}
\# (\cT/\bK^*)\geqslant \exp(1.73738\times d)\,.
\end{equation*}

\end{lemma}


\subsection{A strong  primality criterion}\label{subsection:strong}

Assume that we are in the situation of Section~\ref{section:AKSTEST}.
We  are given an integer $n\geqslant 2$ and we
set $R=\ZZ/n\ZZ$. Let $E$ be an  elliptic curve 
over $R$, let 
$d\geqslant 2001$ be  a prime to $2n$ integer and 
let $T\in E(R)$ be a section of exact order $d$. 
We call $I : E\rightarrow E'$ the quotient by $\langle T\rangle$ 
isogeny  as given by V{\'e}lu's formulae.
Assume we  are given a section $A\in E'_{\text{aff}}(R)$ and  call
\begin{equation*}
\Fgot_A=(x'-x'(A),y'-y'(A))
\end{equation*}
the ideal of $I^{-1}(A)$
in $R[E-E[d]]$.
We assume that  $D(x'(A))$ is a  unit in $R$. 
We call $S=R[x,y,\frac{1}{\psi_d(x)}]/(x'-x'(A),y'-y'(A))$ the 
ring of elliptic periods.
We define the functions $(u_l)_{l\in \ZZ/d\ZZ}$ as in Section~\ref{subsubsection:ENB}. 
There is a unique multiple $\hat T$ of $T$
such that $T=2\HT$.
We set $\eta = u_0(\HT)\in R$.
If $l=2k\bmod d$, we set $\hu_{l}=u_k-\eta$.
We set $\theta_k=u_k\bmod \Fgot_a$ and 
$\htheta_l=\theta_k-\eta$.

Assume now that the following equality  holds true in the ring $S$: 
\begin{equation}\label{eq:congs}
(\htheta_0)^n=\htheta_1\,.
\end{equation}

Let $\hsigma : R[E-\langle T\rangle]\rightarrow R[E-\langle T\rangle]$ be the automorphism induced
on $R[E-\langle T\rangle]$   by  the translation $\tau_{-\HT}$,
\begin{displaymath}
  \begin{array}{crcl}
    \hsigma :& R[E-\langle T\rangle] & \longrightarrow  &   R[E-\langle T\rangle]\,,\\
    &f & \longmapsto  &f\circ\tau_{-T}\,.
  \end{array}
\end{displaymath}
We also denote by $\hsigma : S\rightarrow S$
the  induced map on $S$.
Letting $\hsigma$ repeatedly act on Eq.~(\ref{eq:congs}), 
we deduce that for any
$k\in \ZZ/d\ZZ$, $\htheta_k$ is a power of $\htheta_0$.
In particular, the product  $\prod_k\htheta_k$ is a power of $\htheta_0$.
But Eq.~(\ref{eq:produk}) tells us that 
 this product is a unit in $R$.  So $\htheta_0$ is a unit.
\medskip

Let $p$ be any prime divisor of $n$.
We set $a=\htheta_0 \bmod p \in S/pS$.
We  show  that
the order of $a$
in $(S/pS)^*$ is large.

Let   $\vvv$ be a vector
in  $\cJ\subset \ZZ^d$. 
Let  $(w_k)_{0\leqslant k\leqslant  d-2}$ be
 the coordinates of $\vvv$ in the basis
$(W_k)_{0\leqslant k\leqslant d-2}$ of $\cW$ defined at the end 
of Section~\ref{subsection:units}.
Let $f_{\vvv}=\prod_{0\leqslant k\leqslant d-2}\hu_k^{w_k}$
be the unique multiplicative combination  of the
$\hu_k$ such that  $V(f_\vvv \bmod p)=\vvv$, where
$V$ is  the valuation map
 defined in Eq.~(\ref{eq:V}). 
We note  that $f_{\vvv}\bmod (\Fgot_A,p)=\prod_{0\leqslant k\leqslant d-2} (\htheta_k\bmod p)^{w_k}$
is a power of $a$. Since
$\vvv$ is in $\cJ$, we know   that $f_{\vvv}\bmod p$ has degree $\leqslant (d-1)/2$.
Let  $\vvvone$ 
and $\vvvtwo$ be  two  distinct vectors in $\cJ$.
Let $l_1$ and  $l_2$ be two integers that
are relatively prime to $p$.
Then $l_1f_{\vvvone}\not =l_2f_{\vvvtwo}\bmod (\Fgot_A,p)$ unless $\vvvone=\vvvtwo$
and $l_1=l_2\bmod p$. 
Indeed, if $l_1f_{\vvvone} =l_2f_{\vvvtwo}\bmod (\Fgot_A,p)$ 
then $l_1f_{\vvvone}-l_2f_{\vvvtwo}
\bmod p$ is a function on $E\bmod p$  with degree $\leqslant d-1$
and it cancels
on the degree $d$ divisor $I^{-1}(A)\bmod p$.  So 
$l_1f_{\vvvone} =l_2f_{\vvvtwo}\bmod p$. Therefore,
$f_{\vvvone}$ and $f_{\vvvtwo}$  have the same divisor.
We deduce that $\vvvone=\vvvtwo$. Therefore $l_1=l_2\bmod p$ also.

Using Theorem~\ref{th:AKSgen} and  the lower bound in Lemma~\ref{lemma:bound},
 we deduce the following corollary.
\begin{corollary}[Strong elliptic AKS criterion]\label{cor:AKSS}
Let $n\geqslant 2$  be an integer and let $E$
be an  elliptic curve over $R=\ZZ/n\ZZ$.
Let $T\in E(R)$ be a section of exact order $d$ where
$d\geqslant 2001$ is a prime to $2n$  integer.
 Let $E'$
be the quotient $E/\langle T \rangle$ given by V{\'e}lu's formulae.
Let $A\in E'_{\text{aff}}(R)$ be a section such that the vector
$\ve = \left( \egot_{k}(x'(A)) \right)_{k}$ defined
by Eq.~(\ref{eq:egot}) is 
invertible
for the convolution product $\star$ on $R^d$. Assume the 
congruence 
\begin{equation*}
(\htheta_0)^n=\htheta_1
\end{equation*}
holds true in the ring of elliptic periods $S=R[x,y,{1}/{\psi_d(x)}]/(x'-x'(A),
y'-y'(A))$.

Assume further that 
\begin{equation*}
\exp(1.73738\times d)   \geqslant n^{\sqrt{d}}\,.
\end{equation*}

Then $n$ is a prime power.
\end{corollary}

\begin{appendix}

\section{}
\label{section:appendix}

\subsection{A determinant}\label{subsection:det}

We first compute a determinant that is useful 
in Section~\ref{subsection:units}.
For every integer $n\geqslant 1$, we define $D_n$
to be the determinant of the matrix defined by Eq.~\eqref{eq:vandermonde}.
%
%

We have $D_1=2$ and  $D_2=3$.
We develop the determinant $D_n$ along the first column and
find that 
\begin{math}
D_n=2D_{n-1}-D_{n-2}
\end{math}
for any $n\geqslant 3$.
We deduce, for any $n\geqslant 1$,
\begin{equation}\label{eq:det}
D_n=n+1\,.
\end{equation}

\subsection{Lower bounds for binomial coefficients}\label{subsection:binom}

In this paragraph, we compute effective lower bounds for binomial coefficients.
These estimates will be useful in Section~\ref{subsection:enum}.
Let  $K\geqslant 2$ be  an integer and let $(d_k)_{1\leqslant  k \leqslant K}$ be a family
of positive  integers. We set  $d=\sum_{1\leqslant k \leqslant K}d_k$ and
$\alpha_k=d_k/d$. We set $\valpha= (\alpha_1, \ldots, \alpha_K)$ and define the corresponding entropy to be
\begin{displaymath}
H(\valpha)=H(\alpha_1, \ldots, \alpha_K)=-\alpha_1\log \alpha_1-\alpha_2\log
\alpha_2-\dots - \alpha_K\log \alpha_K.
\end{displaymath}
We recall  Robbins effective Stirling formula \cite{robbins}. For every
positive integer $d$,
\begin{equation*}
\sqrt{2\pi d}\left( \frac{d}{e} \right)^d\exp(\frac{1}{12d+1})\leqslant d!\leqslant \sqrt{2\pi d}\left( \frac{d}{e} \right)^d\exp(\frac{1}{12d})\,.
\end{equation*}
We deduce
\begin{multline*}
(2\pi d)^{\frac{1-K}{2}}
  \exp(d\times H(\alpha_1, \ldots, \alpha_K)+\frac{1}{13}-\frac{K}{12}) \leqslant  \left(\begin{array}{c}{d}\\ d_1 d_2
    \dots d_K \end{array}\right) \leqslant \\ (2\pi d)^{\frac{1-K}{2}}
  \exp(d\times H(\alpha_1, \ldots, \alpha_K)+\frac{1}{12}-\frac{K}{13})\,.
\end{multline*}

We shall need the following definition.

\begin{definition}
Let $\vbeta= (\beta_k)_{1\leqslant k\leqslant K}$  be a family of reals in $]0,1[$ such that
    $\sum_{1\leqslant k\leqslant K} \beta_k=1$. Let  $d$ be a positive
    integer.  
We assume $\beta_k>{1}/{d}$ for every $1\leqslant k\leqslant K$.
For every integer $k$   such that $1\leqslant k\leqslant K-1$,
 set $d_k=\lfloor \beta_k d\rfloor$. We observe that  $d_k$ is positive.
Set $d_{K}=d-\sum_{1\leqslant k\leqslant K-1}d_k$. It is positive also. The {\it rounded } multinomial
coefficient associated to $d$ and $\vbeta$ is defined to be
\begin{displaymath}
\binom d \vbeta = \left(\begin{array}{c}{d}\\ d_1,  d_2, 
    \dots , d_K \end{array}\right).
\end{displaymath}
\end{definition}

In order to find  a nice lower bound for  this coefficient, we set $\alpha_k=d_k/d$ for every $1\leqslant k\leqslant K$.
It is clear that 
\begin{equation*}\label{eq:ab}
\beta_k-\frac{1}{d}\leqslant \alpha_k\leqslant \beta_k\,,
\text{ for $1\leqslant k\leqslant K-1$, and } 
\beta_K\leqslant \alpha_K\leqslant \beta_K+\frac{K}{d}\,.
\end{equation*}

We set $\mu=\max(-\log(\min_{1\leqslant k\leqslant K}(\beta_k-{1}/{d}))-1,1)$ and  we notice that
for any $1\leqslant k\leqslant K$ the derivative of $z\mapsto -z\log z$ is bounded by
$\mu$ in absolute value between $\alpha_k$ and $\beta_k$. Since 
\begin{math}
|\beta_k-\alpha_k|\leqslant {1}/{d}
\end{math}
when $1\leqslant k\leqslant K-1$ and $|\beta_K-\alpha_K|\leqslant {K}/{d}$, we deduce
\begin{equation*}
|H(\alpha_1,\alpha_2,\ldots, \alpha_K)-H(\beta_1,\beta_2,\ldots, \beta_K)|\leqslant \frac{2K\mu}{d}\,.
\end{equation*}
And thus,
\begin{equation}\label{eq:entr}
\frac{1}{d}\log   \binom d \vbeta  \geqslant 
H(\vbeta)-\frac{2K\mu}{d} +\frac{(1-K)\log 2\pi
  d}{2d}+\frac{1}{d}\left(\frac{1}{13}-\frac{K}{12}\right)\,.
\end{equation}

\subsection{An enumeration problem}\label{subsection:enum}

Let $d\geqslant 3$ be an odd integer. We are interested in the set $\cS_d$ of vectors $\ve= (e_1, e_2,
\ldots, e_d)$ in $\ZZ^d$ such that
the sum $\sum_{1\leqslant k\leqslant d} e_k$ of all coordinates is zero
and the $L^1$-norm $\sum_{1\leqslant k\leqslant d}|e_k|$ of $\ve$ is $d-1$.

We look for  a lower  bound for the cardinality of $\cS_d$.
To every  vector $\ve$ in $\cS_d$, we associate a partition of $\{1, 2, \ldots, d\}$
in three sets $E_0$, $E_+$, $E_{-}$ corresponding to the indices with zero,
positive and negative coordinates respectively. 
The sum of positive coordinates equals $(d-1)/2$. 
The sum of negative  coordinates equals $-(d-1)/2$.

We fix a real number $\beta \in ]0,\frac{1}{2}[$ and define the subset 
$\cS_{d,\beta}\subset \cS_d$ consisting of vectors in $\cS$ having
exactly $\lfloor \beta d \rfloor$ positive coordinates and $\lfloor \beta d
\rfloor$ negative coordinates. We assume $\beta d\geqslant 1$.
The number of elements in $\cS_{d,\beta}$ is
\begin{equation}\label{eq:prod3}
\#\cS_{d,\beta} =  \trinom d {\lfloor \beta d \rfloor} {\lfloor \beta d \rfloor} {d-2\lfloor
  \beta d \rfloor } \binom {\frac{d-1}{2}-1}{\lfloor \beta d \rfloor -1 }  \binom
{\frac{d-1}{2}-1}{\lfloor \beta d \rfloor-1} \,.
\end{equation}

The first factor in the product above
is the number of corresponding partitions $E_0\cup E_+\cup E_{-}$. The second
factor is the number of ways one can write $(d-1)/2$ as a sum of $\lfloor
\beta d \rfloor$ strictly positive integers.
The third 
factor is the number of ways one can write $-(d-1)/2$ as a sum of $\lfloor
\beta d \rfloor$ strictly negative integers.

We want to choose the real $\beta$ so as to make the product in
Eq.~(\ref{eq:prod3}) as big as possible.
The logarithm of this product divided by $n$ tends to
$H(\beta,\beta,1-2\beta)+H(2\beta,1-2\beta)$ as $n$ tends to infinity. This expression is maximal
for
$\beta={1}/({2+\sqrt 2})$ and its value is then bigger than $1.7627$.
We set $\beta={1}/({2+\sqrt 2})$ and we  look for an effective lower bound
for every factor in Eq.~(\ref{eq:prod3}).

We first apply Eq.~(\ref{eq:entr}) for $K=3$, $\vbeta
=(\beta,\beta,1-2\beta)$, $\mu =1$, $H(\beta,\beta,1-2\beta)\geqslant 1.08439$
 and $d\geqslant 2001$. We find that
\begin{equation}\label{eq:premfact}
\frac{1}{d}\log \trinom d {\lfloor \beta d \rfloor} {\lfloor \beta d \rfloor} {d-2\lfloor
  \beta d \rfloor }
\geqslant 1.08439-0.00781=1.07658\,.
\end{equation}

We now notice that ${\lfloor \beta d \rfloor}-1\geqslant
\left({(d-1)}/{2}-1\right)/2$ and
$\lfloor \beta'(d-3)\rfloor \geqslant {\lfloor \beta d \rfloor}-1$ provided
${\beta'}/{\beta}\geqslant {d}/({d-3})$, which is guaranteed by setting 
$\beta' =  0.29334$. So,
\begin{equation}\label{eq:2fact}
\binom {\frac{d-1}{2}-1}{\lfloor \beta d \rfloor -1 }\geqslant \binom
       {\frac{d-3}{2}}{\lfloor  2\beta'\frac{d-3}{2} \rfloor  }\,.
\end{equation}
We then apply Eq.~(\ref{eq:entr}) for $K=2$,
$\vbeta=(2\beta',1-2\beta')$, $\mu=1$, $H(2\beta',1-2\beta')\geqslant 0.678$,
and $d\geqslant 999$. We find that
\begin{equation*}
\frac{1}{d}\log \binom d {\lfloor 2\beta' d \rfloor}  
\geqslant 0.678-0.0085=0.6695\,.
\end{equation*}
If we substitute $d$ by $(d-3)/2$ in the above formula, we obtain, for $d\geqslant 2001$,
\begin{equation}\label{eq:3fact}
\frac{1}{d}\log \binom {\frac{d-3}{2}}{\lfloor 2\beta' \frac{d-3}{2} \rfloor}
\geqslant 0.6695\times \frac{d-3}{2d}\geqslant 0.3342\,.
\end{equation}
Combining Eqs.~(\ref{eq:prod3}), (\ref{eq:premfact}), 
(\ref{eq:2fact}), and ~(\ref{eq:3fact}), we deduce the following lemma.

\begin{lemma}\label{lemma:count}
Let $d\geqslant 2001$ be an odd integer and let $\cS_d\subset \ZZ^d$ be 
the set of vectors having  $L^1$-norm equal to $d-1$ and the sum of all
coordinates equal to $0$.
We have
\begin{displaymath}
  \log \#\cS_d\geqslant 1.74498\times d\,.
\end{displaymath}

\end{lemma}


\end{appendix}

\fi

\end{document}